\documentclass[12pt,oneside,leqno]{amsart}
\usepackage{dsfont,amscd,typearea}
\usepackage{mathrsfs}
\usepackage{amsmath,amstext,amsthm}
\usepackage{amssymb}
\usepackage{typearea}
\usepackage{charter}
\usepackage[T1]{fontenc}
\usepackage{esint}
\usepackage{color}

\usepackage[a4paper,width=16.2cm,top=3cm,bottom=3cm]
{geometry}
\numberwithin{equation}{section}
\newtheorem{theorem}{Theorem}[section]
\newtheorem{lemma}[theorem]{Lemma}
\newtheorem{proposition}[theorem]{Proposition}
\newtheorem{corollary}[theorem]{Corollary}

\theoremstyle{definition}
\newtheorem{definition}[theorem]{Definition}

\newtheorem{remark}[theorem]{Remark}

\newtheorem*{ack}{Acknowledgements} 
\numberwithin{equation}{section}

\newcommand{\field}[1]{\mathbb{#1}}

\newcommand{\R}{\field{R}}
\newcommand{\C}{\field{C}}
\newcommand{\N}{\field{N}}

\newcommand{\cali}[1]{\mathscr{#1}}
\newcommand{\cC}{\cali{C}} \newcommand{\cA}{\cali{A}}
\newcommand{\cO}{\cali{O}}

\newcommand{\cR}{\cali{R}} 
\newcommand{\cK}{\cali{K}} \newcommand{\cT}{\cali{T}}

\newcommand{\calig}[1]{\mathcal{#1}}
 
 \newcommand\mH{\calig{H}}
 \newcommand\mO{\calig{O}}
 \newcommand\mR{\calig{R}}
 \newcommand{\cP}{\cali{P}}

\def\bE{{\boldsymbol E}}

\newcommand{\boldsym}[1]{\boldsymbol{#1}}
\newcommand\bb{\boldsym{b}}

\newcommand\bJ{\boldsym{J}}

\newcommand\bt{\boldsym{t}}
\newcommand\bo{\boldsym{o}}
\def\Re{{\rm Re}}
\def\Im{{\rm Im}}

\newcommand{\imat}{\sqrt{-1}}
\DeclareMathOperator{\End}{End}

\DeclareMathOperator{\Ker}{Ker}

\DeclareMathOperator{\rank}{rk}
\DeclareMathOperator{\Id}{Id}
\DeclareMathOperator{\supp}{supp}
\DeclareMathOperator{\tr}{Tr}

\DeclareMathOperator{\vol}{vol}

\DeclareMathOperator{\spec}{Spec}

\newcommand{\norm}[1]{\lVert#1\rVert}
\newcommand{\abs}[1]{\lvert#1\rvert}
\newcommand{\om}{\omega}
\newcommand{\ov}{\overline}
\newcommand{\var}{\varepsilon}

\newcommand{\comment}[1]{}

\newcommand{\hp}{\mathcal{H}_p}
\newcommand{\hpc}{\C {\rm Id}_{\mH_{p}}}
\newcommand{\hph}{\operatorname{Herm}(\mathcal{H}_p)}

\DeclareMathOperator{\dist}{dist}

\allowdisplaybreaks

\begin{document}

\title{Semi-classical properties of Berezin--Toeplitz operators 
with $\cC^k$--\,symbol}

\date{\today}
\author{Tatyana Barron}
\address{Tatyana Barron, Department of Mathematics, 
University of Western Ontario, 
London, Ontario N6A 5B7, Canada}
\email{tatyana.barron@uwo.ca}
\thanks{T.\,B.\ and M.\,P.\ partially supported by the Natural 
Sciences and Engineering Research Council of Canada}
\author{Xiaonan Ma}
\address{Xiaonan Ma, Institut Universitaire de France
\& Universit{\'e} Paris Diderot - Paris 7,
UFR de Math{\'e}matiques, Case 7012,
75205 Paris Cedex 13, France}
\email{ma@math.jussieu.fr}
\thanks{X.\,M.\ partially supported by Institut Universitaire de 
France and the Shanghai Center for Mathematical Sciences}
\author{George Marinescu}
\address{George Marinescu, Universit{\"a}t zu K{\"o}ln,  
Mathematisches Institut, Weyertal 86-90, 50931 K{\"o}ln, Germany
    \& Institute of Mathematics `Simion Stoilow', Romanian Academy,
Bucharest, Romania}
\thanks{G.\,M.\ partially supported by DFG funded projects 
SFB/TR 12, MA 2469/2-1 and CNRS  at Universit\'e Piere et Marie 
Curie, Paris}
\email{gmarines@math.uni-koeln.de}
\author{Martin Pinsonnault}
\address{Martin Pinsonnault, Department of Mathematics, 
University of Western Ontario, 
London, Ontario N6A 5B7, Canada}
\email{mpinson@uwo.ca}
\thanks{}

\begin{abstract}
We obtain the semi-classical expansion of the kernels and 
traces of Toeplitz operators with $\cC^k$--\,symbol on 
a symplectic manifold. We also give a semi-classical estimate 
of the distance of a Toeplitz operator to the space of self-adjoint 
and multiplication operators. 
\end{abstract}
\maketitle
\setcounter{section}{-1}

\tableofcontents


\section{Introduction}\label{s0}
The purpose of this paper is to extend some of the semiclassical 
results about the Berezin-Toeplitz quantization to
the case of Toeplitz operators with $\cC^k$--\,symbol.

A fundamental problem in mathematical physics is to find relations
between classical and 
quantum mechanics. On one side we have symplectic manifolds 
and Poisson algebras, on the other
Hilbert spaces and selfadjoint operators. 
The goal is to establish a dictionary between these theories 
such that given a quantum system depending on a parameter, 
to obtain a classical system when the parameter approaches a 
so called semiclassical limit, in such a way that properties of the
quantum system are controlled up to first order by 
the underlying classical system. It is very interesting to 
go the other way, namely to quantize a classical system, 
that is, introduce a quantum system whose semiclassical limit
is the given classical system.

The aim of the geometric quantization theory of Kostant and Souriau 
is to relate the classical observables (smooth functions) 
on a phase space (a symplectic manifold) to the quantum 
observables (bounded linear operators) on the quantum space 
(holomorphic sections of a line bundle). Berezin-Toeplitz quantization 
is a particularly efficient version of the geometric quantization 
theory \cite{BFFLS,Berez:74,Cha03,Fedo:96,Kos:70,Ma10,Sou:70}.
Toeplitz operators and more generally Toeplitz structures 
were introduced in geometric quantization by
Berezin \cite{Berez:74} and Boutet de Monvel-Guillemin \cite{BoG}. 
Using the analysis of Toeplitz structures 
\cite{BoG}, Bordemann-Meinrenken-Schlichenmaier \cite{BMS} 
and Schlichenmaier \cite{Schlich:00} proved the existence of 
the asymptotic expansion for the composition of Toeplitz operators
in the K\"ahler case when we twist a trivial bundle $E=\C$. 

In \cite{MM07,MM08b} Ma-Marinescu
have extended the Berezin-Toeplitz quantization 
to symplectic manifolds and orbifolds
by using as quantum space the kernel of the Dirac operator 
acting on powers of the prequantum line bundle twisted with 
an arbitrary vector bundle. 
In \cite{MM12} Ma-Marinescu calculated the first coefficients of 
the kernel expansions of Toeplitz operators and of
their composition.

Let us review shortly the results from
\cite{MM07,MM08b,MM11}. We consider a compact 
symplectic manifold $X$ with symplectic form $\om$ 
and a Hermitian line bundle $(L,h^L,\nabla^L)$ whose curvature 
satisfies the prequantization condition  
$\frac{\sqrt{-1}}{2\pi}R^L=\omega$. 
Let $(E,h^E,\nabla^E)$ be a Hermitian vector bundle on $X$
with 
Hermitian connection $\nabla^E$.
Let $J$ be a $\om$-compatible almost-complex structure and
$g^{TX}$ be a $J$-invariant metric on $TX$.
For any $p\in\N$ let $L^p:=L^{\otimes p}$ be the $p^{\text{th}}$ 
tensor power of $L$,
$\Omega^{0,\,\bullet}(X,L^p\otimes E)$ be the space of smooth 
anti-holomorphic forms with values in $L^p\otimes E$ 
with norm induced by $h^L$, $h^E$ and $g^{TX}$, and 
$P_p: \Omega^{0,\,\bullet}(X,L^p\otimes E)\to \Ker(D_p)$
be the orthogonal projection on the kernel of the Dirac operator 
$D_p$.

To any $f\in\cC^\infty(X,\End(E))$
we associate a sequence of linear operators 
\begin{equation}\label{0.0}
T_{f,\,p}: \Omega^{0,\,\bullet}(X,L^p\otimes E)
\to \Omega^{0,\,\bullet}(X,L^p\otimes E), 
\quad T_{f,\,p}=P_p\,f\,P_p\,,
\end{equation} 
where for simplicity we denote by $f$ the operator of 
multiplication with $f$.
For any $f,g\in \cC^\infty(X,\End(E))$,
 the product $T_{f,\,p}\,T_{g,\,p}$ has an asymptotic expansion 
\begin{equation}\label{0.2}
T_{f,\,p}\,T_{g,\,p}=\sum_{k=0}^\infty T_{C_k(f,\,g), p}\, p^{-k}
+\cO(p^{-\infty})
\end{equation}
in the sense of \eqref{atoe2.2}, where $C_k$ are 
bidifferential operators of order $\leqslant 2r$, 
satisfying $C_0(f,g)=fg$ and if $f,g\in  \cC^\infty(X)$,
$C_1(f,g)-C_1(g,f)=\sqrt{-1}\,\{f,g\}$. Here 
$\{ \,\cdot\, , \,\cdot\, \}$ is the Poisson bracket on 
$(X,2\pi \om)$
(cf. \eqref{toe4.1}). We deduce from \eqref{0.2},  
\begin{equation}\label{0.3}
[T_{f,\,p}\,,T_{g,\,p}]=\frac{\sqrt{-1}}{\, p}T_{\{f,g\},\,p}
+\mO(p^{-2})\,.
\end{equation}
Moreover, the norm of the Toeplitz operators allows to recover
the sup-norm of the classical observable
$f\in \cC^\infty(X,\End(E))$:
\begin{equation}\label{0.1}
\lim_{p\to\infty}\norm{T_{f,\,p}}={\norm f}_\infty
:=\sup_{x\in X} |f(x)|\,,
\end{equation}
 and $\|\cdot\|$ is the operator norm.
Thus the Poisson algebra $(\cC^\infty(X),\{ \cdot , \cdot \})$ 
is approximated by the operator algebras of Toeplitz operators 
(for $E=\C$) in the norm sense as $p\to \infty$; the role of 
the Planck constant is played by $\hbar=1/p$. 
This is the so-called semi-classical limit process.

All these papers consider the case of a smooth observable. 
The assumption that the symbol is $\cC^{\infty}$ is quite 
restrictive and analysts studying 
Toeplitz operators normally do not require this. 
While most results on Berezin-Toeplitz quantization 
are for smooth symbols, some progress has been made towards 
understanding what 
is happening with non-smooth symbols, in particular in work 
of L. Coburn and co-authors (see e.g., \cite{BC,C}).
It was remarked recently by Polterovich
\cite{Pol:12} that it is interesting to study the Berezin-Toeplitz 
quantization also for the case of continuous observables. 
A specific example of a situation where it would be helpful to 
know how to quantize non-smooth observables is quantization 
of the universal Teichm\"uller space in work of A. Sergeev 
(see, in particular, \cite{S}).
We will extend in this paper the relations 
\eqref{0.2}, \eqref{0.3}, \eqref{0.1} for $\cC^k$-symbols $f,g$.  
Moreover, we consider the question of how far is a Toeplitz 
operator from being self-adjoint or a multiplication with a function.

In this paper we shall use the kernel calculus of Toeplitz operators 
developed in \cite{MM07,MM08b,MM11,MM12} which 
lends itself very well to handling less regular symbols.

The plan of the paper is as follows. In Section \ref{s1} we recall 
the Bergman kernel expansion \cite{DLM04a,MM07}.
Section \ref{s2} is devoted to the 
kernel expansion of 
the Toeplitz operators. 
In Section \ref{s3} , we explain the expansion of a product of 
Toeplitz operators. In  Section \ref{s4}, we study the
asymptotics of the norm of Toeplitz operators.
Finally, in Section \ref{s5} we consider 
the semi-classical estimates of the distance of a Toeplitz operator 
to various spaces (self-adjoint operators, constant multiples 
of the identity, multiplication operators).
\section{Quantization of symplectic manifolds}\label{s1}
We will briefly describe in this Section the study the Toeplitz 
operators and Berezin-Toeplitz quantization for symplectic manifolds. 
For details we refer the reader to \cite{MM07,MM08b} and 
to the surveys \cite{Ma10,{MM11}}.
We recall in Section \ref{sec:2.1} the definition of the spin$^c$ 
Dirac operator and formulate the spectral gap property for 
prequantum line bundles. In Section \ref{s5.3} we state the 
asymptotic expansion of the Bergman kernel.

\subsection{Spectral gap of the spin$^c$ Dirac operator}
\label{sec:2.1}
We will first show that in the general symplectic case the kernel 
of the spin$^c$ operator is a good substitute for the space 
of holomorphic sections used in K\"ahler quantization. 

Let $(X,\omega)$ be a compact symplectic manifold, 
$\dim_\mathbb{R} X=2n$, with compatible almost complex 
structure $J:TX\to TX$. 
Set $T^{(1,0)}X=\{u\in TX\otimes_{\R}\C:Ju=\imat u\}$ and 
$T^{(0,1)}X=\{u\in TX\otimes_{\R}\C:Ju=-\imat u\}$. 
Let $g^{TX}$ be a $J$-compatible Riemannian metric.
The Riemannian volume form of $g^{TX}$ is denoted by $dv_X$.

We do not assume that $g^{TX}(u,v)=\om(u,Jv)$ for $u,v\in TX$\,.
We relate $g^{TX}$ with $\om$ by means of the skew--adjoint 
linear map ${  \bJ}:TX\longrightarrow TX$ which satisfies 
the relation
\begin{equation} \label{to1.2}
\om(u,v)=g^{TX}({  \bJ}u,v)\quad\text{for}\quad u,v \in TX.
\end{equation}
Then $J$ commutes with ${  \bJ}$, and 
$J={  \bJ}(-{  \bJ}^2)^{-\frac{1}{2}}$.

Let $(L,h^L,\nabla^L)$ be a Hermitian line bundle on $X$, 
endowed with a Hermitian metric $h^L$ and a Hermitian connection 
$\nabla^L$, whose curvature is $R^L=(\nabla^L)^2$. 
We assume that the \emph{prequantization condition} 
\begin{equation}\label{prequantum}
\omega=\frac{\sqrt{-1}}{2\pi}R^{L}\
\end{equation}
is fulfilled.
Let $(E,h^E,\nabla^E)$ be a Hermitian vector bundle on $X$ with 
Hermitian metric $h^E$ and Hermitian connection $\nabla^E$.
We will be concerned with asymptotics in terms of high tensor 
powers $L^p\otimes E$, when $p\to\infty$,
that is, we consider the semi-classical limit $\hbar=1/p\to 0$.

Let us denote by 
\begin{align}\label{eq:1.1}
\bE:=\Lambda^{\bullet}(T^{*(0,1)}X)\otimes E 
\end{align}
the bundle of anti-holomorphic forms with values in $E$. 
The metrics $g^{TX}$, $h^L$ and $h^E$ induce an $L^2$-scalar 
product on $\Omega^{0,\,\bullet}(X,L^p\otimes E)
=\cC^{\infty}(X,L^p\otimes\bE)$ by
\begin{equation}\label{lm2.0}
\big\langle s_1,s_2 \big\rangle =\int_X\big\langle s_1(x),
s_2(x)\big\rangle_{L^p\otimes\bE}\,dv_X(x)\,,
\end{equation} 
whose completion is denoted 
$\big(L^2(X,L^p\otimes\bE),\|\cdot\|_{L^2}\big)$.

Let $\nabla^{\rm det}$ be the connection on $\det(T^{(1,0)}X)$ 
induced by the projection of the Levi-Civita connection 
$\nabla^{TX}$ on $T^{(1,0)}X$.
Let us consider the Clifford connection $\nabla^{\text{Cliff}}$ 
on $\Lambda^{\scriptscriptstyle{\bullet}}(T^{*(0,1)}X)$ 
associated to $\nabla^{TX}$ and to the connection 
$\nabla^{\rm det}$ on $\det(T^{(1,0)}X)$ 
(see e.g.\ \cite[\S\,1.3]{MM07}). The connections $\nabla^L$, 
$\nabla^E$ and $\nabla^{\text{Cliff}}$ induce the connection 
\[
\nabla_p=\nabla^{\text{Cliff}}\otimes\operatorname{Id}
+\operatorname{Id}\otimes\nabla^{L^p\otimes E}\quad
\text{on $\Lambda^{\bullet}(T^{*(0,1)}X)\otimes L^p\otimes E$.}
\]
The \emph{spin$^c$ Dirac operator} is defined by 
\begin{equation}\label{defDirac}
D_p=\sum_{j=1}^{2n}\mathbf{c}(e_j)\nabla_{p,\,e_j}:
\Omega^{0,\,\bullet}(X,L^p\otimes E)\longrightarrow
\Omega^{0,\,\bullet}(X,L^p\otimes E)\,,
\end{equation}
where $\{e_j\}_{j=1}^{2n}$ local orthonormal frame of $TX$ 
and $\mathbf{c}(v)=\sqrt{2}(v^\ast_{1,0}\wedge-i_{v_{\,0,1}})$ 
is the Clifford action of $v\in TX$. Here we use the decomposition 
$v=v_{\,1,0}+v_{\,0,1}$, 
$v_{\,1,0}\in T^{(1,0)}X$, $v_{\,0,1}\in T^{(0,1)}X$ 
and $v^\ast_{1,0}\in T^{*(0,1)}X$ is the dual of $v_{1,0}$.
 
\begin{remark}\label{R:kaehler}
Let us assume for a moment that $(X,J)$ is a complex manifold 
(i.\,e., $J$ is integrable) and the bundles $L$, $E$ are 
holomorphic and $\nabla^L$, $\nabla^E$ are the Chern connections. 
If $g^{TX}(u,v)=\om(u,Jv)$ for 
$u,v\in TX$ (thus $(X,g^{TX})$ is K\"ahler), then 
\begin{equation}\label{lm2.0b}
D_p=\sqrt{2}\big(\,\overline{\partial}
+\overline{\partial}^{\,*}\big),\quad \overline{\partial}
=\overline{\partial}^{L^p\otimes E}\,,
\end{equation}
so $\Ker(D_p)=H^0(X,L^p\otimes E)$ for $p\gg1$. 
The following result (Theorem \ref{specDirac}) shows that for 
a general symplectic manifold $\Ker(D_p)$ has all semi-classical 
properties of $H^0(X,L^p\otimes E)$.

Note that if $(X,J)$ is a complex manifold but $g^{TX}$ is not 
necessarily associated to $\omega$ by $g^{TX}(u,v)=\om(u,Jv)$, 
we can still work with the operator 
$\widetilde{D}_p:=\sqrt{2}\big(\overline{\partial}
+\overline{\partial}^{\,*}\big)$
instead of $D_p$ (cf.\ \cite[Theorem\,1.5.5]{MM07}), although 
$\widetilde{D}_p$ is only a modified Dirac operator. 
Theorem \ref{specDirac} 
and the results which follow remain valid for $\widetilde{D}_p$, 
so that, for $p$ large, the quantum space will be 
$H^0(X,L^p\otimes E)$ ($=\Ker(\widetilde{D}_p)$ for $p\gg1$). 
\end{remark}
Let us return now to our general situation of a compact symplectic 
manifold $(X,\omega)$, endowed with a $\omega$--\,compatible 
almost complex structure $J$ and $J$--\,compatible Riemannian 
metric $g^{TX}$. Let
\begin{equation}\label{main1a}
    \mu_{0}=\inf\big\{R^L_x(u,\overline{u})/|u|^2_{g^{TX}}:
u\in T^{(1,0)}_xX\setminus\{0\},x\in X\big\}\,.
\end{equation}
By \eqref{prequantum} we have $\mu_0>0$.
\begin{theorem}[{\cite[Th.\,1.1,\,2.5]{MM02}, 
    \cite[Th.\,1.5.5]{MM07}}]\label{specDirac}
There exists $C>0$ such that for any $p\in\mathbb{N}$ and any 
$s\in\bigoplus_{k>0}\Omega^{0,k}(X,L^p\otimes E)$, we have
\begin{equation}\label{main1}
\Vert D_{p}s\Vert^2_{L^2}
\geqslant(2\mu_0 p-C)\Vert s\Vert^2_{L^2}\, .
\end{equation}
Moreover, 
the spectrum of $D^2_p$ verifies 
\begin{align}\label{diag5}
&\spec (D^2_p)\subset \big\{0\big\}
\cup \big[2\mu_0 p-C,+\infty\big[\,.
\end{align}
\end{theorem}
\noindent
The proof of Theorem \ref{specDirac} is based on a 
direct application of the Lichnerowicz formula for $D_p^2$\,.

By the Atiyah-Singer index theorem we have for $p\gg1$,
\begin{equation}\label{as}
\dim \Ker(D_p)=\int_{X}\operatorname{Td}(T^{(1,0)}X)
\operatorname{ch}(L^{p}\otimes E)
=\rank (E)\,\frac{p^n}{n!}\int_X \omega^n+O(p^{n-1}) \,, 
\end{equation}
where $\operatorname{Td}$ is the Todd class and 
$\operatorname{ch}$ is the Chern character.
Theorem \ref{specDirac} shows the forms in $\Ker(D_p)$ 
concentrate asymptotically in the $L^2$ sense on 
their zero-degree component and \eqref{as} shows that 
$\dim\Ker(D_p)$ is a polynomial 
in $p$ of degree $n$, as in the holomorphic case. 

\subsection{Off-diagonal asymptotic expansion of Bergman kernel}
\label{s5.3}
We recall that a bounded linear operator $T$ on 
$L^2(X,L^p\otimes \bE)$ is called
Carleman operator (see e.g., \cite{HalSu}) if there exists 
a kernel $T(\cdot\,,\cdot)$ such that 
$T(x,\cdot)\in L^2(X,(L^p\otimes\bE)_{x}
\otimes (L^p\otimes\bE)^{*})$ and 
\begin{equation}\label{lm2.01c}
(TS)(x)=\int_{X}T(x,x') S(x') dv_{X}(x')\,,\quad
\text{for all $S\in L^2(X,L^p\otimes\bE)$\,.}
\end{equation}
Let us introduce the orthogonal projection
\[
P_p:L^2(X,L^p\otimes\bE)\longrightarrow\Ker(D_p),
\] called the Bergman projection in analogy to the K\"ahler case. 
It is a Carleman operator whose integral kernel is called 
\emph{Bergman kernel}\/. 
Set $d_p := \dim \Ker(D_p)$.
Let $\{S^p_i\}_{i=1}^{d_p}$ be any orthonormal basis of
$\Ker(D_p)$ with respect to the inner
product \eqref{lm2.0}. Then
\begin{equation} \label{bk2.4a}
P_p(x,x')= \sum_{i=1}^{d_p} S^p_i (x) \otimes (S^p_i(x'))^*
\in (L^p\otimes\bE)_{x}\otimes (L^p\otimes\bE)_{x'}^*\,.
\end{equation}
The \emph{Toeplitz operator} with symbol 
$f\in L^{\infty}(X,\End(E))$ is defined by
\begin{equation} \label{bk2.4}
T_{f,p}:L^2(X,L^p\otimes\bE)\to L^2(X,L^p\otimes\bE)\,,
\quad T_{f,p}= P_{p}fP_{p}\,,
\end{equation}
where the action of $f$ is the pointwise multiplication by $f$.
The map which associates to  $f\in L^{\infty}(X,\End(E))$
the family of bounded operators $\{T_{f,\,p}\}_p$ on 
$L^2(X,L^p\otimes\bE)$ is called
the  {\em Berezin-Toeplitz quantization}\/.
Note that $T_{f,\,p}$ is a Carleman operator with smooth integral 
kernel given by
\begin{equation}\label{toe2.5}
T_{f,\,p}(x,x')=\int_X P_p(x,x'')f(x'')P_p(x'',x')\,dv_X(x'')\,.
\end{equation}
The existence of the spectral gap expressed in 
Theorem \ref{specDirac} allows us to \emph{localize} the behavior 
of the Bergman kernel and of the kernel of Toeplitz operators.

Let $a^X$ be the injectivity radius of $(X, g^{TX})$.
We denote by $B(x,\var)=B^{X}(x,\var)$ and  
$B(0,\var)=B^{T_xX}(0,\var)$ the open balls
in $X$ and $T_x X$ with center $x$ and radius $\var$, respectively.
Then the exponential map $ T_x X\ni Z \to \exp^X_x(Z)\in X$ is a
diffeomorphism from $B^{T_xX} (0,\var)$ onto $B^{X} (x,\var)$ for
$\var\leqslant a^X$.  {}From now on, we identify $B^{T_xX}(0,\var)$
with $B^{X}(x,\var)$ via the exponential map for
$\var \leqslant a^X$.
Throughout what follows, $\varepsilon$ runs in the fixed 
interval $]0, a^X/4[$. 


\noindent
By \cite[Proposition\,4.1]{DLM04a}, we have the {\em far
off-diagonal} behavior of the Bergman kernel: 
\begin{proposition}\label{0t3.0}
For any $\ell,m\in\N$ and $\var>0$, there exists 
$C_{\ell,m,\var}>0$ such that for any $p\geqslant 1$, $x,x'\in X$, 
$d(x,x')> \var$, 
\begin{equation}\label{1c3}
\left|P_p(x,x')\right|_{\cC^m(X\times X)} 
\leqslant C_{\ell,\,m,\:\var}\, p^{-l}\,,
\end{equation}
The $\cC ^m$ norm \eqref{1c3} is induced by
$\nabla^L$, $\nabla^E$, $h^L$, $h^E$ and $g^{TX}$.
\end{proposition}

Let $\pi : TX\times_{X} TX \to X$ be the natural projection from 
the fiberwise product of $TX$ on $X$.
  Let $\nabla ^{\End (\bE)}$ be the connection on
$\End (\Lambda (T^{*(0,1)}X)\otimes E)$ induced by
$\nabla ^{\text{Cliff}}$ and $\nabla ^E$.

Let us elaborate on the identifications we use in the sequel.
Let $x_0\in X$ be fixed and consider 
the diffeomorphism $B^{T_{x_0}X}(0,4\varepsilon)\ni Z 
\mapsto \exp^X_{x_0}(Z)\in B^{X}(x_0,4\varepsilon)$
for $\var\leqslant a ^X/4$.
We denote the pull-back of the vector
bundles $L$, $E$ and $L^p\otimes\bE$ via this diffeomorphism 
by the same symbols. 
\\[3pt]
\textbf{(i)} There exist trivializations of $L$, $E$ and 
$L^p\otimes\bE$ over 
$B^{T_{\!x_0}X} (0,4\varepsilon)$ given by unit frames 
which are parallel with respect to $\nabla ^L$, $\nabla^E$ and
$\nabla^{L^p\otimes\bE}$ along the curves 
$\gamma_Z:[0,1]\to B^{T_{\!x_0}X}(0,4\varepsilon)$ defined 
for every $Z\in B^{T_{\!x_0}X}(0,4\varepsilon)$ by
$\gamma_Z(u)=\exp^X_{x_0} (uZ)$.
\\[3pt]
\textbf{(ii)} With the previous trivializations, 
$P_p(x,x')$ induces a smooth section 
\[B^{T_{\!x_0}X}(0,4\var)\ni Z,Z'\mapsto P_{p,\,x_0}(Z,Z')\]
of $\pi^*(\End(\bE))$ over $TX\times_{X} TX$, which depends 
smoothly on $x_0$.
\\[3pt]
\textbf{(iii)} $\nabla ^{\End (\bE)}$ induces naturally 
a $\cC^m$-norm with respect to the parameter $x_0\in X$.
\\[3pt]
\textbf{(iv)} 
By \eqref{to1.2}, $\bJ$ is an element of\/ $\End(T^{(1,0)}X)$.  
Consequently, we can diagonalize $\bJ_{x_0}$\/, i.e., choose 
an orthonormal basis $\{w_j\}_{j=1}^n$ of $T^{(1,0)}_{x_0}X$ 
such that 
\begin{equation} \label{3ue60}
\bJ_{x_0}\om_j =\frac{\sqrt{-1}}{2\pi} a_j(x_0) w_j\,, 
\quad\text{for all $j=1,2,\ldots,n$}\,,
\end{equation}
where $0<a_1(x_0)\leqslant a_2(x_0)
\leqslant\dotsc\leqslant a_n(x_0)$. 
The vectors $\{e_{j}\}_{j=1}^{2n}$ defined by
\begin{equation}\label{frame1}
e_{2j-1}=\tfrac{1}{\sqrt{2}}(w_j+\overline{w}_j)
\quad\text{and}\quad
e_{2j}=\tfrac{\sqrt{-1}}{\sqrt{2}}(w_j-\overline{w}_j)\,, 
\quad j=1,\dotsc,n\,,
\end{equation}
form an orthonormal frame of $T_{\!x_0}X$.
The diffeomorphism 
\begin{equation}\label{alm4.22}
\R^{2n}\ni(Z_1,\dotsc, Z_{2n}) \longmapsto \sum_i
Z_i e_i\in T_{x_0}X
\end{equation}
induces coordinates on $T_{x_0}X$, which we use throughout 
the paper. In these coordinates we have 
$e_j={\partial}/{\partial Z_j}$\,.
%
%
%
The complex coordinates $z=(z_1,\ldots,z_n)$ on $T_{\!x_0}X$ 
are given by $z_j=Z_{2j-1}+\imat Z_{2j}$, $j=1,\dotsc,n$.
\\[3pt]
\textbf{(v)} If $dv_{TX}$ is the Riemannian volume form 
on $(T_{\!x_0}X, g^{T_{\!x_0}X})$, there exists 
a smooth positive function $\kappa_{x_0}:T_{\!x_0}X\to\R$, 
$Z\mapsto\kappa_{x_0}(Z)$ defined by 
\begin{equation} \label{atoe2.7}
dv_X(Z)= \kappa_{x_0}(Z)\, dv_{TX}(Z),\quad \kappa_{x_0}(0)=1,
\end{equation}
where the subscript $x_0$ of  
$\kappa_{x_0}(Z)$ indicates the base point $x_0\in X$. 
By \cite[(4.1.101)]{MM07} we have
\begin{equation} \label{atoe2.8}
\kappa_{x_0}(Z)=1+O(|Z|^2)\,.
\end{equation}
\noindent
\textbf{(vi)} 
Let $\Theta_p: L^2(X,L^p\otimes\bE)\longrightarrow 
L^2(X,L^p\otimes\bE)$ 
be a sequence of continuous linear  operators with smooth kernel 
$\Theta_p(\cdot,\cdot)$ with respect to $dv_X$ 
(e.g., $\Theta_p=T_{f,p}$).
In terms of our basic trivialization, $\Theta_p(x,y)$ induces 
a family of smooth sections
$Z,Z'\mapsto \Theta_{p,\,x_0}(Z,Z^\prime)$ of 
$\pi^*\End(\bE)$ over $\{(Z,Z')\in TX\times_{X} TX:|Z|,|Z'|<\var\}$, 
which depends smoothly on $x_0$. 

We denote by $\abs{\Theta_{p,\,x_0}(Z,Z^\prime)}_{\cC^m(X)}$ 
the $\cC^{m}$ norm with respect to the parameter $x_0\in X$. 
We say that
\[
\Theta_{p,\,x_0}(Z,Z^\prime)=\mO(p^{-\infty})\,,\quad p\to\infty,
\]
 if
for any $\ell,m\in \N$, there exists $C_{\ell,\,m}>0$ such that 
$\abs{\Theta_{p,\,x_0}(Z,Z^\prime)}_{\cC^{m}(X)}
\leqslant C_{\ell,\,m}\, p^{-\ell}$.

We denote by $\det_{\C}$ for the determinant function on the 
complex bundle $T^{(1,0)}X$ and set 
$|\bJ_{x_0}|=(-\bJ^2_{x_0})^{1/2}$. Note that
\begin{equation}\label{toe1.3a}
    {\det}_{\C} |\bJ_{x_0}|=\prod_{i=1}^n\frac{a_i(x_0)}{2\pi}\,,
\end{equation}
where $a_i(x_0)$ were defined in \eqref{3ue60}.
Let
\begin{equation}\label{toe1.3}
\begin{split}
\cP_{x_0}(Z,Z')&={\det}_{\C}|\bJ_{x_0}|\,
\exp\Big (- \frac{\pi}{2} 
\big\langle |\bJ_{x_0}|(Z-Z'),(Z-Z') \big\rangle
-\pi \sqrt{-1} \big\langle \bJ_{x_0} Z,Z' \big\rangle\Big )\\
&=\frac{1}{(2\pi)^n}\prod_{i=1}^n
a_i(x_0)\:\:\exp\Big(-\frac{1}{4}\sum_i
a_i(x_0)\big(|z_i|^2+|z^{\prime}_i|^2 -2z_i\overline{z}_i'\big)\Big)\,.
\end{split}
\end{equation}
We recall that $\cP_{x_0}(Z,Z')$ is actually the kernel of the 
orthogonal projection $\cP_{x_0}$
from $L^2(\R^{2n})$ onto the Bargmann-Fock space, 
see \cite[\S2]{MM08b}. Especially, $\cP_{x_0}^2=\cP_{x_0}$ 
so that
\begin{equation}\label{toe1.4}
\cP_{x_0}(Z,Z')=\int_{\R^{2n}}\cP_{x_0}(Z,Z'')
\cP_{x_0}(Z'',Z')\,dv(Z'')\,.
\end{equation}
Fix $k\in\N$ and $\var^\prime\in ]0,a^X[$\,. 
Let 
$\big\{Q_{r,\,x_0}\in \End(\bE)_{x_0}[Z,Z^{\prime}]:
0\leqslant r\leqslant k,x_0\in X\big\}$
be a family
of polynomials in $Z,Z^ \prime$, which is smooth with respect to
the parameter $x_0\in X$. We say that
\begin{equation} \label{toe2.7}
p^{-n} \Theta_{p,x_0}(Z,Z^\prime)\cong \sum_{r=0}^k
(Q_{r,\,x_0} \cP_{x_0})(\sqrt{p}Z,\sqrt{p}Z^{\prime})p^{-r/2}
+\mO(p^{-(k+1)/2})\,,
\end{equation}
on $\{(Z,Z^\prime)\in TX\times_X TX:\abs{Z},\abs{Z^{\prime}}
<\var^\prime\}$ 
if there exist a decomposition
\begin{equation} \label{toe2.71}
\begin{split}
p^{-n} \Theta_{p,x_0}(Z,Z^\prime)\kappa^{1/2}_{x_0}(Z)
\kappa^{1/2}_{x_0}(Z')-\sum_{r=0}^k
(Q_{r,\,x_0} \cP_{x_0})(\sqrt{p}Z,\sqrt{p}Z^{\prime})p^{-r/2}\\
=R_{p,k,x_0}(Z,Z^\prime)+\mO(p^{-\infty})\,,
\end{split}
\end{equation}
where $R_{p,k,x_0}$ satisfies the following estimate on 
$\big\{(Z,Z^\prime)\in TX\times_X TX:\abs{Z},\abs{Z^{\prime}}
<\var^\prime\big\}$:
for every $m,m'\in\N$ there exist $C_{k,\,m,\,m'}>0$, $M>0$ 
such that for all $p\in\N^{*}$
\begin{equation}
\begin{split}
\label{toe2.711}
\sup_{|\alpha|+|\alpha'|\leqslant m'}
\Biggl|\frac{\partial^{\alpha+\alpha'}}{\partial Z^\alpha\partial 
Z'^{\alpha'}}&R_{p,k,x_0}(Z,Z^\prime)\Biggr|_{\cC^m(X)}\\ \\
\leqslant  
&\,C_{k,\,m,\,m'}\,p^{(m'-k-1)/2}
(1+\sqrt{p}\,|Z|+\sqrt{p}\,|Z^{\prime}|)^M \,
e^{-C_0\,\sqrt{ p}\,|Z-Z^{\prime}|}\,.
\end{split}
\end{equation}
We consider the orthogonal projection:
\begin{equation}\label{2c3}
I_{\C\otimes E}: \bE=\Lambda^\bullet (T^{*(0,1)}X)\otimes E
\longrightarrow \C\otimes E\,.
\end{equation}
By \cite[Theorem 4.18$^\prime$]{DLM04a} we have 
the following {\em off diagonal expansion} of the Bergman kernel:
\begin{theorem} \label{tue17}
Let $\varepsilon\in\,]0, a^X/4[$. There exist a smooth family 
relative to the parameter $x_0\in X$ 
\[
\Big\{J_{r,\,x_0}(Z,Z^{\prime})\in
\End (\bE_{x_0})[Z,Z^{\prime}]:r\in\N,\,x_0\in X\Big\}\,,
\quad \deg J_{r,\,x_0}\leqslant 3r\,,
\]
of polynomials $J_{r,\,x_0}$ having the same parity as $r$, and
whose coefficients
are polynomials in $R^{TX}$, $R^{T^{(1,0)}X}$,
$R^E$ {\rm (}and $R^L${\rm )}  and their derivatives of order
$\leqslant r-1$  {\rm (}resp\,.\ $\leqslant r${\rm )} such that 
\begin{equation} \label{toe1.9}
p^{-n} P_{p,\,x_0}(Z,Z^\prime)\cong \sum_{r=0}^k
(J_{r,\,x_0} \cP_{x_0})\,(\sqrt{p}Z,\sqrt{p}Z^{\prime})
p^{-\frac{r}{2}}
+\mO(p^{-\frac{k+1}{2}})\,,
\end{equation}
on the set $\big\{(Z,Z^\prime)\in TX\times_X TX:\abs{Z},
\abs{Z^{\prime}}<2\var\big\}$. Moreover 
$J_{0,\,x_0}(Z,Z^{\prime})=I_{\C\otimes E}$\,.
\end{theorem}
By taking $Z=Z'=0$ in \eqref{toe1.9} we obtain the diagonal 
expansion of the Bergman kernel. Namely, 
for any $k\in \N$, $P_p(x,x)$ has an 
expansion in the $\cC^\infty$-topology 
\begin{equation}\label{toe1.9b}
    P_p(x,x)=\sum_{r=0}^k \bb_r(x)p^{n-r} + \cO(p^{n-k-1})\,,
\quad \bb_r\in\cC^\infty(X,\End(\bE))\,,
\end{equation}
and by Theorem \ref{tue17} and (\ref{toe1.3}), we get
\begin{equation}\label{toe2.712}
\bb_0(x_0)={\det}_{\C}|\bJ_{x_0}|\,I_{\C\otimes E}\in
\End(\bE_{x_0})\,.
\end{equation}
Let us remark that if $g^{TX}(u,v)=\om(u,Jv)$ for $u,v\in TX$, 
then $a_i=2\pi$, so 
$\bb_0(x)=I_{\C\otimes E}$. 

\section{Expansion of the kernels and traces of Toeplitz operators}
\label{s2}
For a smooth symbol $f\in\cC^{\infty}(X,\End(E))$ we know 
by \cite[Lemma\,4.6]{MM08b}, \cite[Lemma 7.2.4]{MM07} that 
the kernel of the associated 
Toeplitz operator $T_{f,p}$ as in \eqref{bk2.4} has for any 
$\ell\in\N$ an expansion on the diagonal 
in the $\cC^\infty$-topology,
\begin{equation}\label{toe2.1}
T_{f,p}(x,x)=\sum _{r=0}^\ell \bb_{r,f}(x)p^{n-r}
+\cO(p^{n-\ell-1})\,,\quad\bb_{r,f}\in\cC^{\infty}(X,\End(\bE))\,,
\end{equation}
where
\begin{equation}\label{toe2.2}
\bb_{0,f}(x)=\bb_0(x)f(x)\,.
\end{equation}
Note as an aside, that the coefficients $\bb_{r,f}$, $r=0,1,2$, 
were calculated in \cite[Theorem\,0.1]{MM12}, if $(X,\omega)$
is K\"ahler and the bundles $L$, $E$ are holomorphic.
In \cite[Lemma\,4.6]{MM08b} we actually established 
the off-diagonal expansion of the Toeplitz kernel. 
We wish to study here the asymptotic behavior of the 
Toeplitz kernel for a less regular symbol $f$. 
Let us begin with the analogue of \cite[Lemma\,4.2]{MM08b}.
\begin{lemma} \label{toet2.1}
Let $f\in L^{\infty}(X,\End(E))$.
For every $\varepsilon>0$ and every $\ell,m\in\N$, 
there exists $C_{\ell,m,\varepsilon}>0$ such that 
\begin{equation} \label{toe2.6b}
|T_{f,\,p}(x,x')|_{\cC^m(X\times X)}
\leqslant C_{\ell,m,\varepsilon}p^{-\ell}\,,
\quad\text{for all $p\geqslant 1$, $(x,x')\in X\times X$ 
with $d(x,x')>\varepsilon$},
\end{equation}
where the $\cC^m$-norm is induced by $\nabla^L,\nabla^E$ and
$h^L,h^E,g^{TX}$.
Moreover, there exists  $C>0$ such that
for all $p\geqslant 1$ and all $(x,x')\in X\times X$
with $d(x,x')\leqslant\varepsilon$,
\begin{equation} \label{toe2.6c}
|T_{f,\,p}(x,x')|_{\cC^m(X\times X)}\leqslant Cp^{n+\frac{m}2}
e^{-\frac{1}{2}C_{0}\,\sqrt{ p}\, d(x,x')} + O(p^{-\infty}).
\end{equation}
\end{lemma} 
\begin{proof}
Due to \eqref{toe2.5} and \eqref{1c3},  \eqref{toe2.6b} holds 
if we replace $T_{f,\,p}$ by $P_p$\,. 
Moreover, from \eqref{toe1.9}, for any $m\in \N$, 
there exists $C_m>0$ such that 
\begin{equation}\label{toe2.6d}
|P_p(x,x')|_{\cC^m(X\times X)}<Cp^{n+\frac{m}{2}}\,,
\quad \text{for all $(x,x')\in X\times X$}.
\end{equation}
These two facts and formula \eqref{toe2.5} imply \eqref{toe2.6b}.
By Theorem \ref{tue17}, for $Z,Z'\in T_{x_0}X$, $|Z|,|Z'|<\varepsilon$ we have
\begin{equation}\label{toe2.6e}
\begin{split}
\sup_{|\alpha|+|\alpha'|\leqslant m}
\Biggl|\frac{\partial^{\alpha+\alpha'}}{\partial Z^\alpha\partial Z'^{\alpha'}}
&P_{p,x_0}(Z,Z^\prime)\Biggr|
  \\ \\
&\leqslant\,C\,p^{n+\frac{m}2}
(1+\sqrt{p}\,|Z|+\sqrt{p}\,|Z^{\prime}|)^M \,
e^{-C_0\,\sqrt{ p}\,|Z-Z^{\prime}|}+O(p^{-\infty})\,.
\end{split}
\end{equation}
Using \eqref{toe2.5} and \eqref{toe2.6e} by taking $x_0=x$,
we get \eqref{toe2.6c}.
\end{proof}
In order to formulate our results for a family of functions 
(observables) we consider equicontinuous and uniformly bounded
families. The results are of course valid for individual functions, too. 
\begin{definition} \label{toed2.2}
    We denote  by $\nabla^{E}$ the connection on 
    $(T^*X) ^{\otimes k}\otimes \End(E)$ induced by 
    the Levi-Civita connection $\nabla ^{TX}$ and $\nabla ^E$.
Let $\cA^0 \subset \cC^0(X, \End(E))$ 
be a subset which is equicontinuous on $X$,
and $\cA^1 \subset \cC^1(X, \End(E))$ 
be a subset such that $\nabla^{E}\cA^1\subset \cC^0(X, \End(E))$
is uniformly bounded. Let $\cA^2 \subset \cC^2(X, \End(E))$ 
be a subset such that $\nabla^{E}\cA^2\subset \cC^0(X, \End(E))$
is uniformly bounded, and 
$\nabla^{E}\nabla^{E}\cA^2\subset \cC^0(X, \End(E))$
is equicontinuous on $X$.
Let $\cA^k _{\infty}$ be a subset of $\cA^k$ (for $k=0,1,2$)
which is uniformly bounded.
 \comment{For $k\in \N$, let $\cA^k \subset \cC^k(X, \End(E))$ 
 be a subset of $\cC^k(X, \End(E))$ such that 
 $\underbrace{\nabla^{E}\cdots \nabla^{E}}_{k-\text{times}}\cA^k
 \subset \cC^0(X, \End(E))$ is uniformly continuous on $X$. }
\end{definition}    
\begin{theorem} \label{toet2.2}
Let $(X,\omega)$ be a compact symplectic manifold, 
$(L,h^L,\nabla^L)\to X$ a prequantum line bundle, 
$(E,h^E,\nabla^E)\to X$ be an auxiliary vector bundle.
We have as $p\to\infty$
\begin{align}
&p^{-n}T_{f,p}(x,x)=f(x)\bb_{0}(x)+o(1) (\|f\|_{\cC ^0}+1)\,,\:\:
\text{uniformly for } f\in \cA^0, x\in X
\,,\label{toe2.7a}\\
&p^{-n}T_{f,p}(x,x)=f(x)\bb_{0}(x)+O(p^{-1/2})(\|f\|_{\cC ^0}+1)\,,
\:\:\text{uniformly for } f\in \cA^1, x\in X
\,,\label{toe2.7b}\\
&p^{-n}T_{f,p}(x,x)=f(x)\bb_{0}(x)+{\bb}_{1,f}(x)p^{-1}
+o(p^{-1})(\|f\|_{\cC ^0}+1)\,, \label{toe2.7c}\\
&\hspace{40mm}\quad\text{uniformly for } f\in \cA^2, x\in X.
\nonumber
\end{align}
In particular, the remainders $o(1)$, $O(p^{-1/2})$, $o(p^{-1})$ 
do not depend on $f$.
\end{theorem}
\begin{proof}
We start by proving \eqref{toe2.7a}. 
Recall that we trivialized $L, E$ by a unit frame 
over $B^{T_{\!x}X}(0,4\var)$ which is parallel with respect to 
$\nabla^L, \nabla^E$ along the geodesics starting in $x$.
With this trivialization, the section $f\in\End(E)$ induces 
a  section 
\[B^{T_{\!x}X}(0,4\var)\ni Z\mapsto f_{x}(Z).\]
We denote by $f(x)\in\End(E)|_U$ the endomorphism obtained 
by parallel transport of $f(x)\in\End(E_{x})$ into the neighboring 
fibers $\End(E_{x'})$, $x'\in U=B(0,4\var)$.

Let $\delta>0$ be given. Since
$\cA ^0$ is uniformly continuous on $X$, 
there exists $\varepsilon>0$ such that 
$B(x,\varepsilon)\subset U$ and for all $x'\in B(x,\varepsilon)$ 
we have $|f(x')-f(x)|\leqslant\delta$ for any $f\in \cA ^0$. 

By \eqref{toe2.5} and \eqref{toe2.6b} we have
\begin{equation}\label{eq3.11}
\begin{split}
T_{f,\,p}(x,x)
=\int_{B(x,\varepsilon)} P_p(x,x')f(x')P_p(x',x)\,dv_X(x')
+O(p^{-\infty}) \|f\|_{\cC^0}\,.
\end{split}
\end{equation}
We write now $f(x')=f(x)+(f(x')-f(x))\in \End(E_{x'})$ 
and split accordingly the last integral in a sum of two integrals. 
From Theorem \ref{tue17}, the first one is 
\begin{multline}\label{toe2.8a}
\int_{B(x,\varepsilon)} P_p(x,x')f(x)P_p(x',x)\,dv_X(x')\\
= p^{2n} \int_{B(x,\varepsilon)} \sum_{i+j=0} ^1 
(J_{i,x}\cP_{x})(0, \sqrt{p}Z')f(x) 
(J_{j,x}\cP_{x})(\sqrt{p}Z',0) dZ'+ O(p^{n-1}) \|f\|_{\cC^0}.
\end{multline}

Note that $J_{0,x_{0}}=I_{\C\otimes E}$ and 
$J_{1,x_{0}}(0, Z')$ is a polynomial on $Z'$ with odd degree,
thus 
\begin{multline}\label{eq3.12}
     \int_{\R ^{2n}} 
(J_{1,x}\cP_{x})(0, Z')f(x) 
(J_{0,x}\cP_{x})_{x}(Z',0) dZ'\\
= \int_{\R ^{2n}} 
(J_{0,x}\cP_{x})(0, Z')f(x) 
(J_{1,x}\cP_{x})(Z',0) dZ' =0.
    \end{multline}
    From (\ref{toe2.712}), (\ref{toe2.8a}) and (\ref{eq3.12}), 
    we get
\begin{equation}\label{toe3.13}
\begin{split}
\int_{B(x,\varepsilon)}& P_p(x,x')f(x)P_p(x',x)\,dv_X(x')
=f(x)\cP_{x}(0,0)p^{n}+O(p^{n-1}) \|f\|_{\cC^0}\\
&=f(x)\bb_0(x)p^n+O(p^{n-1}) \|f\|_{\cC^0}\,.
\end{split}
\end{equation}
Note that if $f$ is a function, then we have
$$
\int_{B(x,\varepsilon)} P_p(x,x')f(x)P_p(x',x)\,dv_X(x')
=f(x) \Big(P_p(x,x')+O(p^{-\infty})\Big) \,.
$$
The second one can be estimated by
\begin{equation}\label{toe2.8b}
\begin{split}
\Big|\int_{B(x,\varepsilon)}& P_p(x,x')\big(f(x')-f(x)\big)
P_p(x',x)\,dv_X(x')\Big|\\
&\leqslant
\delta\int_{B(x,\varepsilon)} |P_p(x,x')||P_p(x',x)|\,dv_X(x')\\
&=\delta\int_{B(0,\varepsilon)} |P_{p,x}(0,Z')||P_{p,x}(Z',0)|\,dv_X(Z')\, .
\end{split}
\end{equation}
We use now the off-diagonal expansion from Theorem \ref{tue17}. 
By \eqref{toe1.9} we have 
\begin{equation}\label{toe2.8ba}
\begin{split}
&P_{p,x}(Z,0)=p^{n}\Big(\textstyle\prod_{j=1}^n\frac{a_j}{2\pi}\,
e^{-\frac{1}{4}p\sum_{j=1}^na_j|z_j|^2}
+p^{-1/2}R_{p,x}(Z,0)+O(p^{-\infty})\Big)
\kappa_{x}^{-\frac{1}{2}}(Z)\,\\
&\quad\text{where}\quad|R_{p,x}(Z,0)|
\leqslant\,C(1+\sqrt{p}\,|Z|)^M\,e^{-C_0\,\sqrt{ p}\,|Z|}\,,
\end{split}
\end{equation}
so in order to estimate the last integral in \eqref{toe2.8b},
by (\ref{atoe2.7}),
we have to estimate
\[
\int_{B(0,\varepsilon)} p^{2n}
\Big|\textstyle\prod_{j=1}^n\frac{a_j}{2\pi}\,
e^{-\frac{1}{4}p\sum_{j=1}^na_j|z'_j|^2}+p^{-1/2}R_{p,x}(Z',0)
+O(p^{-\infty})\Big|^2 \,dZ'\,.
\]
By using the change of variables $\sqrt{p}Z'=Y$ we see that
\[
p\int_{\C}e^{-\pi p|Z'|^2}\,dZ'=1\,,
\quad p^n\int_{\C^n}(1+\sqrt{p}\,|Z'|)^M\,
e^{-C_0\,\sqrt{ p}\,|Z'|}\,dv_{X}(Z')=O(1)\,,
\]
hence
\[
\int_{B(0,\varepsilon)} |P_{p,x}(0,Z')||P_{p,x}(Z',0)|\,dv_{X}(Z')
=O(p^n)\,,
\]
so there exists $C>0$ such that for any 
$x\in X$, $f\in \cA ^0$, $p\in \N$, we have
\begin{equation}\label{toe2.8c}
\Big|\int_{B(x,\varepsilon)} P_p(x,x')\big(f(x')-f(x)\big)
P_p(x',x)\,dv_X(x')\Big|\leqslant C\delta\,p^n\,.
\end{equation}
From \eqref{eq3.11}, \eqref{toe3.13}
and  \eqref{toe2.8c}, we get \eqref{toe2.7a}.

\comment{Assume that $f\in\cC^1(X,\End(E))$. 
Recall that we trivialized $E$ by a unit frame 
over $B^{T_{\!x}X}(0,4\var)$ which is parallel with respect to 
$\nabla^E$ along the geodesics starting in $x$.
With this trivialization, the section $f\in\End(E)$ induces 
a smooth section 
\[B^{T_{\!x}X}(0,4\var)\ni Z\mapsto f_{x}(Z)\]
of $\pi^*(\End(E))$ over $TX\times_{X} TX$, which depends 
smoothly on $x$.  }

By Taylor's formula, there exist
$C>0, \varepsilon>0$, such that for $|Z|\leqslant\varepsilon$,
$f\in \cA ^1$, we have
\begin{equation}\label{toe2.8d}
f_{x}(Z)-f_{x}(0)=
R(Z)\,,\quad |R(Z)|\leqslant C |Z|\,.
\end{equation} 
We repeat the proof above by plugging this expression in the 
integral from \eqref{toe2.8c},
\comment{
We observe that 
\[\int_{\C^n}e^{-\pi p\sum_{j=1}^na_j|Z_j|^2}Z_j\,dZ=0\] 
so that only }
we observe that only
\[
\int_{B(0,\varepsilon)} |P_{p,x}(0,Z')||Z'||P_{p,x}(Z',0)|\,dv_{X}(Z')
\] 
contributes to the subleading term.
But then the change of variables $\sqrt{p}Z'=Y$ introduces a factor
$p^{-1/2}$, whereof \eqref{toe2.7b} follows.

Finally,
for any $\delta>0$ there exists $\varepsilon>0$, such that for 
$|Z|\leqslant\varepsilon$, $f\in \cA ^2$,
\begin{equation}\label{toe2.8e}
f_{x}(Z)-f_{x}(0)=\textstyle\sum_{j}\partial_j f_{x}(0)Z_j
+\sum_{j,k}\partial_{jk} f_{x}(0)Z_jZ_k+R(Z)
\,,\quad |R(Z)|\leqslant\delta|Z|^2\,.
\end{equation} 
Taking into account the proof of the
asymptotic expansion \eqref{toe2.1} 
from \cite[Lemma 4.6]{MM08b}, we see that \eqref{toe2.7c} holds.
\end{proof}
\begin{remark}
In the same vein, we show that in the conditions of 
Theorem \ref{toet2.2}, we have for $f\in\cC^k(X,\End(E))$, $k\in\N$, as $p\to\infty$, 
\begin{equation}
p^{-n}T_{f,p}(x,x)=\sum_{r=0}^{\lfloor k/2\rfloor}\bb_{r,f}(x)p^{-r}
+R_{k,p}(x)\,,\:\text{uniformly on $X$,}
\end{equation}
where $\bb_{r,f}$ are the universal coefficients 
from \eqref{toe2.1} and
\begin{equation}
R_{k,p}=\begin{cases}
o(p^{-k/2})\,,&\quad\text{for $k$ even}\,,\\
O(p^{-k/2})\,,&\quad\text{for $k$ odd}\,.
\end{cases}
\end{equation}
Here $\lfloor a\rfloor$ denotes the integer part of $a\in\R$.
\end{remark}
We recall that by \cite[(4.79)]{MM08b}, \cite[(7.4.6)]{MM07}, 
for any $f,g\in\cC^\infty(X,\End(E))$, the kernel of 
the composition $T_{f,\,p}\circ T_{g,\,p}$
has for all $\ell\in\N$ an asymptotic expansion on the diagonal in 
the $\cC^\infty$-topology,
\begin{equation}\label{toe4.30}
(T_{f,\,p}\,\circ T_{g,\,p})(x,x)
=\sum_{r=0}^{\ell} \bb_{r,\,f,\,g}(x) p^{n-r}+\cO(p^{n-\ell-1})\,,
\quad  \bb_{r,\,f,\,g}\in \cC^\infty(X,\End(\bE))\,.
\end{equation}
The coefficients $\bb_{r,\,f,\,g}$, $r=0,1,2$, were calculated in 
\cite[Theorem\,0.2]{MM12} in the case of a K\"ahler manifold 
$(X,\omega)$ and of holomorphic bundles $L$ and $E$. 
We give here the analogue of the expansion \eqref{toe4.30}
in the case of $\cC^k$ symbols.
\begin{theorem} \label{toet2.3}
Let $m\in\N$ and $f_1,\ldots,f_m\in L^{\infty}(X,\End(E))$. 
Write
\begin{equation}\label{toe2.9}
p^{-n}(T_{f_1,p}\ldots T_{f_m,p})(x,x)=f_1(x)\ldots f_m(x)\bb_0(x)
+R_p(x)\,.
\end{equation}
We have as $p\to\infty$, uniformly in $x\in X$,
\begin{align}
R_p(x)&=o(1),\: \text{uniformly on } f_{i}\in \cA_{\infty} ^0\, ,
\label{toe2.9a}\\
R_p(x)&=O(p^{-1/2}),\: 
\text{uniformly on } f_{i}\in \cA_{\infty} ^1\, ,
\label{toe2.9b}\\
R_p(x)&=O(p^{-1}),\: \text{uniformly on } f_{i}\in \cA_{\infty} ^2. 
\label{toe2.9c}
\end{align}
\end{theorem}
\begin{proof} To prove \eqref{toe2.9a} let $\delta>0$ be given. 
    Choose $\varepsilon>0$ such that for
$x'\in B(x,\varepsilon)$ we have $|f_j(x')-f_j(x)|\leqslant\delta$, 
$1\leqslant j\leqslant m$, where $f_j(x)\in\End(\bE_{x'})$
is the parallel transport of $f_j(x)\in\End(\bE_{x})$ 
as in the previous proof. By \eqref{bk2.4} we have 
$T_{f_1,p}\ldots T_{f_m,p}=P_pf_1P_pf_2\ldots P_pf_mP_p$ 
hence 
\begin{equation}
\begin{split}
&(T_{f_1,p}\ldots T_{f_m,p})(x,x)\\
&=\int_{X^m} P_p(x,x_1)f_1(x_1)P_p(x_1,x_2)f_{2}(x_2)\ldots 
f_m(x_m)P_p(x_m,x)\textstyle\prod_{i=1}^mdv_X(x_i)\\
&=I+O(p^{-\infty})\,,
\end{split}
\end{equation}
where
\begin{multline*}
I=\int\limits_{\stackrel{|Z_i|\leqslant\varepsilon}{1
\leqslant i\leqslant m}} 
P_{p,x}(0,Z_1)f_{1,x}(Z_1)P_{p,x}(Z_1,Z_2)f_{2,x}(Z_2)
\ldots f_{m,x}(Z_m)P_{p,x}(Z_m,0)
\textstyle\prod_{i=1}^mdv_X(Z_i).
\end{multline*}
We write now
\[
I=I_0+\sum_{j=1}^{m}I_j
\]
with
\begin{align}\label{toe2.11}
\begin{split}
I_0=\int\limits_{\stackrel{|Z_i|\leqslant\varepsilon}{1
\leqslant i\leqslant m}}& P_p(0,Z_1)f_{1,x}(0)P_p(Z_1,Z_2)
f_{2,x}(0)\ldots f_{m,x}(0)P_p(Z_m,0)
\textstyle\prod_{i=1}^mdv_X(Z_i)\\
&=f_1(x)\ldots f_m(x)\big(\cP_{x}(0,0)p^{n}+O(p^{n-1})\big)\,,
\end{split}
\end{align}
in the above second equation, we use the argument 
in \eqref{eq3.12}, and for $1\leqslant j\leqslant m$,
\[
\begin{split}
I_j=\int\limits_{\stackrel{|Z_i|\leqslant\varepsilon}
{1\leqslant i\leqslant m}} P_p(0,Z_1)&f_{1,x}(Z_1)
\ldots P_p(Z_{j-1},Z_{j})\big(f_{j,x}(Z_{j})-f_{j,x}(0)\big)\\
&P_p(Z_{j},Z_{j+1})f_{j+1,x}(Z_{j+1})\ldots f_{m,x}(Z_m)P_p(Z_m,0)
\textstyle\prod_{i=1}^mdv_X(Z_i).
\end{split}
\]
By \eqref{toe1.9},  for $1\leqslant j\leqslant m$, we have
\begin{equation}\label{toe2.10}
|I_j|\leqslant C\delta\int\limits_{\stackrel{|Z_i|\leqslant\varepsilon}
{1\leqslant i\leqslant m}}
|P_p(0,Z_1)|A_p|P_p(Z_m,0)|\textstyle\kappa ^{1/2}_{x}(Z_{1})
\kappa ^{1/2}_{x}(Z_{m})\prod_{i=1}^md Z_i
\end{equation}
where
\[
A_p=p^{(m-1)n}\prod_{i=1}^{m-1}(1+\sqrt{p}\,|Z_i|
+\sqrt{p}\,|Z_{i+1}|)^M \,
e^{-C_0\,\sqrt{ p}\,|Z_i-Z_{i+1}|}\,.
\]
We plug now the expansion \eqref{toe2.8ba} for $P_p(0,Z_1)$ 
and $P_p(Z_m,0)$ in \eqref{toe2.10}
and estimate the exponential terms appearing there.
By (\ref{main1a}), we have
\[
\begin{split}
&\exp(-\textstyle\frac{p}{4}\sum_{j=1}^na_j|Z_{1j}|^2-C_0\,
\sqrt{p}\,|Z_1-Z_{2}|)\\
&\hspace{20mm}\leqslant 
\exp(-\tfrac{\mu_{0}}{8} p|Z_1|^2)\,
\exp(-C(\sqrt{p}|Z_2|-\tfrac14)),
\end{split}
\]
since
\[
\begin{split}
&\tfrac{\mu_{0}}{8} p|Z_1|^2+C_0\,\sqrt{p}\,|Z_1-Z_{2}|\geqslant
C\big(p|Z_1|^2+\sqrt{p}\,|Z_1-Z_{2}|\big)\\
&\hspace{20mm}\geqslant
C\big(\sqrt{p}|Z_1|-\tfrac14+\sqrt{p}\,|Z_1-Z_{2}|\big)\geqslant
C\big(\sqrt{p}|Z_2|-\tfrac14\big)\,.
\end{split}
\]
We pair now one factor $e^{-\frac{C}{2}\sqrt{p}|Z_2|}$ with
$e^{-C_0\,\sqrt{ p}\,|Z_2-Z_{3}|}$
and obtain
\[
e^{-\frac{C}{2}\sqrt{p}|Z_2|}\,e^{-C_0\,\sqrt{ p}\,|Z_2-Z_{3}|}
\leqslant
e^{-C_3\sqrt{p}|Z_3|}=e^{-\frac{C_3}{2}\sqrt{p}|Z_3|}\,
e^{-\frac{C_3}{2}\sqrt{p}|Z_3|}\,;
\]
we pair further $e^{-\frac{C_3}{2}\sqrt{p}|Z_3|}$ with 
$e^{-C_0\,\sqrt{ p}\,|Z_3-Z_{4}|}$ and so on. 
Finally, we obtain for the left-hand side of \eqref{toe2.10} 
the estimate
\begin{equation}\label{toe2.10a}
\begin{split}
&\int\limits_{\stackrel{|Z_i|\leqslant\varepsilon}
{1\leqslant i\leqslant m}}
|P_p(0,Z_1)|A_p|P_p(Z_m,0)|\textstyle\kappa ^{1/2}_{x}(Z_{1})
\kappa ^{1/2}_{x}(Z_{m})\prod_{i=1}^md Z_i\\ 
&\leqslant\int\limits_{\stackrel{|Z_i|\leqslant\varepsilon}
{1\leqslant i\leqslant m}}
\!\exp\!\Big(\!-\frac{\mu_0}{8} p|Z_1|^2-C_2 \sqrt{p}|Z_2|-
\ldots-C_{m-1} \sqrt{p}|Z_{m-1}|-\frac{\mu_0}{4} p|Z_m|^2\Big)
p^{2n}B_p\textstyle\prod_{i=1}^mdZ_i
\end{split}
\end{equation}
where
\[
B_p=p^{(m-1)n}(1+\sqrt{p}\,|Z_1|)^M(1+\sqrt{p}\,|Z_m|)^M
\prod_{i=1}^{m-1}(1+\sqrt{p}\,|Z_i|+\sqrt{p}\,|Z_{i+1}|)^M \,.
\]
Since the right-hand side integral in \eqref{toe2.10a} converges 
we obtain that $|I_j|\leqslant C'\delta p^n$, for some $C'>0$. 
This completes the proof of \eqref{toe2.9a}.

To prove \eqref{toe2.9b} and \eqref{toe2.9c} we repeat 
the proof above by estimating
$f_{j,x}(Z_{j})-f_{j,x}(0)$ with the help of Taylor 
formulas \eqref{toe2.8d} and \eqref{toe2.8e}. As in the proof of 
Theorem \ref{toet2.2} we obtain the remainders $O(p^{-1/2})$ and
$O(p^{-1})$, respectively, due to the change of variables 
$\sqrt{p}Z=Y$.
\end{proof}
\begin{remark}
In the same vein, we show that if $f,g\in\cC^{k}(X,\End(E))$, 
$k\in\N$, we have as $p\to\infty$, 
\begin{equation}
(T_{f,\,p}\,\circ T_{g,\,p})(x,x)
=\sum_{r=0}^{\lfloor k/2\rfloor} \bb_{r,\,f,\,g}(x) p^{n-r}
+R_{k,p}(x)\,,\:\text{uniformly on $X$,}
\end{equation}
where $\bb_{r,f,g}$ are the universal coefficients 
from \eqref{toe4.30} and
\begin{equation*}
R_{k,p}=\begin{cases}
o(p^{-k/2})\,,&\quad\text{for $k$ even}\,,\\
O(p^{-k/2})\,,&\quad\text{for $k$ odd}\,.
\end{cases}
\end{equation*}
\end{remark}
We will now consider traces of Toeplitz operators.
\begin{theorem}
Let $f\in L^{\infty}(X,\End(E))$. Then for any $k\in\N$ we have as 
$p\to\infty$,
\begin{equation}\label{toe4.1}
\tr(T_{f,p})=\sum_{r=0}^k\bt_{r,f}p^{n-r}+O(p^{n-k-1})\,,
\quad\text{with $\bt_{r,f}=\int_X\tr[\bb_rf]\,dv_X$}\,.
\end{equation}
\end{theorem}
\begin{proof}
By \eqref{toe1.9b},  we infer
\begin{equation}
\begin{split}
\tr(T_{f,p})&=\tr(P_pfP_p)=\tr(P_pf)
=\int_X\tr\big[P_p(x,x)f(x)\big]\,dv_X(x)\\
&=
\sum_{r=0}^k p^{n-r}\int_X\tr\big[\bb_r(x) f(x)\big]dv_X(x)
+O(p^{n-k-1})\,. 
\end{split}
\end{equation}
\end{proof}
\begin{theorem}\label{toet2.5}
Let $f_1,\ldots,f_m\in L^{\infty}(X,\End(E))$. Write
\begin{equation}\label{eq3.30}
p^{-n}\tr(T_{f_1,p}\ldots T_{f_m,p})
=\int_X\tr\big[f_1\ldots f_m\big]\,\frac{\omega^n}{n\,!}+R_p\,.
\end{equation}
Then as $p\to\infty$,
\begin{equation}\label{eq3.31}
R_p=\begin{cases}
o(1)\,,&\quad \text{uniformly on } f_{i}\in \cA_{\infty} ^0\, , \\
O(p^{-1/2})\,,&\quad
\text{uniformly on } f_{i}\in \cA_{\infty} ^1\, ,\\
O(p^{-1})\,,&\quad \text{uniformly on } f_{i}\in \cA_{\infty} ^2\, .
\end{cases}
\end{equation}
\end{theorem}
\begin{proof}
We have
\[\tr(T_{f_1,p}\ldots T_{f_m,p})
=\int_X\tr(T_{f_1,p}\ldots T_{f_m,p})(x,x)\,dv_X\,,\]
and we apply Theorem \ref{toet2.3} together with the 
dominated convergence theorem.
\end{proof}
When $(X,J,\om)$ is a compact K\"ahler manifold, 
$g^{TX}(\cdot,\cdot)=\omega(\cdot,J\cdot)$, 
$E=\C$ with the trivial metric, and each
$f_{i}\in \cC ^{\infty}(X)$, then
\eqref{eq3.30} appears in \cite[p.\,292]{BMS}, 
\cite[Th.\,4.2]{BPU} with $R_p=O(p^{-1})$.
\section{Expansion of a product of Toeplitz operators}\label{s3}
We consider in this section the expansion of the composition of
two Toeplitz operators at the operator level. We recall first the 
situation for Toeplitz operators with smooth symbols. 
A {\em Toeplitz operator}\index{Toeplitz operator} 
is a sequence $\{T_p\}=\{T_p\}_{p\in\N}$ of linear operators
$T_{p}:L^2(X,L^p\otimes\bE)\longrightarrow 
L^2(X,L^p\otimes\bE)$
with the properties:
\begin{itemize}
\item[(i)] For any $p\in \N$, we have 
$T_{p}=P_p\,T_p\,P_p$\,.
\item[(ii)] There exist a sequence $g_l\in\cC^\infty(X,\End(E))$ 
such that for all $k\geqslant0$ there exists $C_k>0$ 
such that for all $p\in \N^*$, we have
\begin{equation}\label{toe2.3}
\Big\|T_p-P_p\Big(\sum_{l=0}^k p^{-l}g_l\Big) P_p\Big\|
\leqslant C_k\,p^{-k-1},
\end{equation}
where $\norm{\cdot}$ denotes the operator norm on the space of 
bounded operators.
\end{itemize}
We write symbolically
\begin{equation}\label{atoe2.2}
T_p= P_p\Big(\sum_{l=0}^\infty p^{-l}g_l\Big) P_p
+\mO(p^{-\infty}).
\end{equation}

Let $f,g\in\cC^\infty(X,\End(E))$. By \cite[Theorem\,1.1]{MM08b}, 
the product of the Toeplitz operators 
$T_{f,\,p}$ and  $T_{g,\,p}$ is a Toeplitz operator, 
more precisely, it admits the asymptotic expansion
in the sense of \eqref{atoe2.2}{\rm:}
\begin{equation}\label{toe4.2}
T_{f,\,p}\circ T_{g,\,p}=\sum^\infty_{r=0}p^{-r}T_{C_r(f,g),\,p}
+\mO(p^{-\infty}),
\end{equation} 
where $C_r$ are bidifferential operators with smooth coefficients
of total degree $2r$ (cf. \cite[Lemma 4.6, (4.80)]{MM08b}). We have $C_0(f,g)=fg$ and  
if $f,g\in\cC^\infty(X)$, 
\begin{equation}\label{toe4.2b}
C_1(f,g)-C_1(g,f)= \sqrt{-1}\{f,g\} \Id_E.
\end{equation} 
In the case of a K\"ahler manifold $(X,\omega)$, the operators 
$C_0,C_1,C_2$ were calculated in \cite[Theorem\,0.1]{MM12}.

We study now the expansion of the product of two Toeplitz 
operators with $\cC^k$ symbols.
\begin{theorem}\label{toet2.4}
Let $k\in\N$ and $f,g\in\cC^{k}(X,\End(E))$. Then for 
$m\in\{0,\ldots,\lfloor k/2\rfloor\}$, we have
\begin{equation}\label{toe2.10b0}
T_{f,\,p}\,\circ T_{g,\,p}=\sum_{r=0}^m p^{-r}\,T_{C_r(f,g),p}
+R_{m,p}\,,
\end{equation}
where $C_r(f,g)$ are the universal coefficients from \eqref{toe4.2} 
and $R_{m,p}$ satisfies the following estimates: 
\begin{equation}\label{toe2.10c0}
R_{m,p}=\begin{cases}
o(p^{-k/2})\,,&\quad\text{for $m=\lfloor k/2\rfloor$}\,,\\
O(p^{-m-1})\,,&\quad\text{for $m<\lfloor k/2\rfloor$}\,,
\end{cases}
\end{equation}
in the operator norm sense.
\end{theorem}
In order to prove this theorem we need to develop some machinery 
from \cite{MM08b} concerning a criterion for a sequence of 
operators to be a (generalized) Toeplitz operator. For 
this purpose we refine the condition from \textbf{(vi)} in Section 1. 

Let $\Theta_p: L^2(X,L^p\otimes\bE)\longrightarrow 
L^2(X,L^p\otimes\bE)$ 
be a sequence of continuous linear  operators with smooth kernel 
$\Theta_p(\cdot,\cdot)$ with respect to $dv_X$.
Fix $k\in\N$ and $\var^\prime\in ]0,a^X[$\,. 
Let 
\[\big\{Q_{r,\,x_0}\in\End(\bE)_{x_0}[Z,Z^{\prime}]:
0\leqslant r\leqslant k,x_0\in X\big\}\]
be a family of polynomials in $Z,Z^ \prime$, such that 
$Q_{r,\,x_0}$ is of class $\cC^{k-r}$ with respect to
the parameter $x_0\in X$. 
We say that
\begin{equation} \label{toe2.73}
p^{-n} \Theta_{p,x_0}(Z,Z^\prime)\cong \sum_{r=0}^k
(Q_{r,\,x_0} \cP_{x_0})(\sqrt{p}Z,\sqrt{p}Z^{\prime})p^{-r/2}
+\mO(p^{-\frac{k+1}{2}})\,,
\end{equation}
on $\{(Z,Z^\prime)\in TX\times_X TX:\abs{Z},\abs{Z^{\prime}}
<\var^\prime\}$ 
if there exist a decomposition (\ref{toe2.71}) and (\ref{toe2.711})
holds for $m=m'=0$.

We say that
\begin{equation} \label{toe2.732}
p^{-n} \Theta_{p,x_0}(Z,Z^\prime)\cong \sum_{r=0}^k
(Q_{r,\,x_0} \cP_{x_0})(\sqrt{p}Z,\sqrt{p}Z^{\prime})p^{-r/2}
+\bo(p^{-\frac{k}{2}})\,,
\end{equation}
if there exist a decomposition \eqref{toe2.71}
where $R_{p,k,x_0}$ satisfies the following estimate: for any 
$\delta>0$, there exists $\varepsilon>0$, $C_{k}>0$, $M>0$ 
such that for all $(Z,Z^\prime)\in TX\times_X TX$ with 
$\abs{Z},\abs{Z^{\prime}}
<\var$ and $p\in\N^{*}$,
\begin{equation}
\label{toe2.7113}
\Big|R_{p,k,x_0}(Z,Z^\prime)\Big|_{\cC^0(X)}
\leqslant  \,\delta\,p^{-k/2}
(1+\sqrt{p}\,|Z|+\sqrt{p}\,|Z^{\prime}|)^M \,
e^{-C_0\,\sqrt{ p}\,|Z-Z^{\prime}|}\,.
\end{equation}
We have the following analogue of \cite[Lemma\,4.6]{MM08b}.
\begin{lemma} \label{toel2.3}
Let $f\in\cC^k(X,\End(E))$.
There exists a family 
\[
\big\{Q_{r,\,x_0}(f)\in\End(\bE)_{x_0}[Z,Z^{\prime}]:
0\leqslant r\leqslant k,x_0\in X\big\}
\] 
such that
\\[2pt]
(a) $Q_{r,\,x_0}(f)$
are polynomials with the same parity as $r$, 
\\[2pt]
(b) $Q_{r,\,x_0}(f)$ is of class $\cC^{k-r}$ with respect to
the parameter $x_0\in X$,
\\[2pt]
(c) There exists  $\varepsilon\in ]0, a^X/4[$
such that for any $m\in\{0,1,\ldots,k\}$, $x_0\in X$,
 $Z,Z^\prime \in T_{x_0}X$, $\abs{Z},\abs{Z^{\prime}}<\var/2$, 
 we have 
\begin{equation} \label{toe2.13}
p^{-n}T_{f,\,p,\,x_0}(Z,Z^{\prime})
\cong \sum^m_{r=0}(Q_{r,\,x_0}(f)\cP_{x_0})
(\sqrt{p}Z,\sqrt{p}Z^{\prime})
p^{-\frac{r}{2}} + \mR_{m,p}\,,
\end{equation}
where
\[
\mR_{m,p}=
\begin{cases}
\mO(p^{-\frac{m+1}{2}})\,,&\text{if $m\leqslant k-1$}\,,\\
\bo(p^{-\frac{m}{2}})\,,&\text{if $m=k$}\,.
\end{cases}
\]
in the sense of \eqref{toe2.73} and \eqref{toe2.732}.
The coefficients $Q_{r,\,x_0}(f)$ are expressed by 
\begin{equation} \label{toe2.14}
Q_{r,\,x_0}(f) = \sum_{r_1+r_2+|\alpha|=r}
  \cK\Big[J_{r_1,\,x_0}\;,\; 
\frac{\partial ^\alpha f_{\,x_0}}{\partial Z^\alpha}(0) 
\frac{Z^\alpha}{\alpha !} J_{r_2,\,x_0}\Big]\,.
\end{equation}
Especially,
\begin{align} \label{toe2.15}
Q_{0,\,x_0}(f)= f(x_0)\,I_{\C\otimes E} .
\end{align}
\end{lemma}
\begin{proof}
We just have to modify the proof of \cite[Lemma\,4.6]{MM08b}
in what concerns the Taylor formula for $f_{x_0}$:
\begin{equation*}
f_{x_0}(Z)=\sum_{|\alpha|\leqslant m}
\frac{\partial^\alpha f_{x_0}}{\partial Z^\alpha}(0)
\frac{ Z^\alpha}{\alpha!}+\cR_m(Z)\,,\quad \cR_{m}(Z)=
\begin{cases}
O(|Z|^{m+1})\,,&\text{if $m\leqslant k-1$}\,,\\
o(|Z|^{m})\,,&\text{if $m=k$}\,.
\end{cases}
\end{equation*}
thus
\begin{equation}\label{toe2.18}
f_{x_0}(Z)=\sum_{|\alpha|\leqslant m} p^{-|\alpha|/2}
\,\frac{\partial^\alpha f_{x_0}}{\partial Z^\alpha}(0)
\frac{ (\sqrt{p}Z)^\alpha}{\alpha!}
+\cR_{m,p}(Z)\,.
\end{equation}
where
\[
\cR_{m,p}(Z)=
\begin{cases}
p^{-\frac{m+1}{2}}O(|\sqrt{p}Z|^{m+1})\,,&\text{if 
$m\leqslant k-1$}\,,\\
o(p^{-\frac{m}{2}}) O(|\sqrt{p}Z|^{m})\,,&\text{if $m=k$}\,.
\end{cases}
\]
The last line just means that there exists $C>0$ such that for any 
$\delta>0$, there exists $\varepsilon>0$ such that for all
$|Z|\leqslant\varepsilon$ and all $p\in\N$, we have 
$|\cR_{m,p}(Z)|\leqslant C\delta p^{-\frac{m}{2}}|\sqrt{p}Z|^{m}$.
\end{proof}

\begin{lemma}\label{toet3.4}
    Let $\cT_p: L^2(X,L^p\otimes\bE)\longrightarrow 
L^2(X,L^p\otimes\bE)$ 
be a sequence of continuous linear  operators with smooth kernel 
$\cT_p(\cdot,\cdot)$ with respect to $dv_X$.
Assume that in the sense of (\ref{toe2.732}), 
\begin{align}\label{b6.34}
 p^{-n} \cT_{p,x_0} (0,Z^{\prime})
 \cong\bo(1)\,,\:p\longrightarrow\infty\,.  
\end{align}
 Then there exists 
$C>0$ such that for every $\delta>0$, there exists $p_0$ 
such that for every $p>p_0$ and
$s\in L^2(X,L^p\otimes\bE)$ we have
\begin{gather}
\label{b6.342}
\norm{\cT_p\,s}_{L^2}\leqslant C \delta\norm{s}_{L^2}\,,
\quad \|\cT_{p}^* s  \|_{L^2} \leqslant C \delta \|s\|_{L^2}\,.
\end{gather}
\end{lemma}
\begin{proof}
By the Cauchy-Schwarz inequality we have
\begin{equation}\label{b6.381}
\big\|\cT_p\,s\big\|_{L^2} ^2
\leqslant \int_{X} \Big(\int_{X}
\big|\cT_p(x,y)\big| dv_{X}(y)\Big)
 \Big(\int_{X}
\big|\cT_p(x,y)\big| |s(y)|^2
dv_{X}(y)\Big)dv_{X}(x)\,.
\end{equation}
We split then the inner integrals into integrals over 
$B^X(x,\var^\prime)$ 
and $X\setminus B^X(x,\var^\prime)$ and use the fact that 
the kernel of $\cT_p$
has the growth $\mO(p^{-\infty})$ outside the diagonal. 
By \eqref{b6.34}, there exists $C'>0$ such that for every 
$\delta>0$, there exists $p_0$ such that for every $p>p_0$ 
and $x\in X$, 
\begin{align}\label{b6.383}
    \begin{split}
\int_{X}\big|\cT_p(y,x)\big| dv_{X}(y)&
\leqslant \int_{B^X(y,\var^\prime)} C p^n \delta 
(1+\sqrt{p}d(y,x))^M e^{-C_{0}\sqrt{p}d(y,x)} dv_{X}(y)
+ O(p^{-\infty})\\
&=O(1)\delta+ O(p^{-\infty}),\\	
\int_{X}\big|\cT_p(x,y)\big| dv_{X}(y)&
\leqslant C\delta\, .
\end{split}\end{align}
Combining \eqref{b6.381} and \eqref{b6.383} and Fubini's theorem 
we obtain
\begin{equation}\label{b6.385}
\begin{split}
\big\|\cT_p\,s\big\|_{L^2} ^2
&\leqslant   C\delta \int_X\Big(\int_{X}\big|\cT_p(x,y)\big| 
|s(y)|^2 dv_{X}(y)\Big)dv_{X}(x)   \\
&=  C\delta\int_X\Big(\int_{X}\big|\cT_p(x,y)\big|  
dv_{X}(x)\Big)|s(y)|^2dv_{X}(y)\\
&\leqslant (C\delta)^2\int_X|s(y)|^2dv_{X}(y).
\end{split}
\end{equation}
This proves the first estimate of \eqref{b6.342}. 
The second one follows by taking the adjoint. 
The proof of Lemma \ref{toet3.4} is completed. 
\end{proof}

\begin{proof}[Proof of Theorem \ref{toet2.4}]
Firstly, it is obvious that $P_p\,T_{f,\,p}\,T_{g,\,p}\,P_p
=T_{f,\,p}\,T_{g,\,p}$. 
Lemmas \ref{toet2.1} and \ref{toel2.3} imply
that for $Z,Z^\prime \in T_{x_0}X$, 
$\abs{Z},\abs{Z^{\prime}}<\var/4$:
\begin{multline}\label{toe4.5}
(T_{f,\,p}\circ T_{g,\,p})_{x_0}(Z,Z^\prime)= \int_{{ T_{x_0}X}}
T_{f,\,p,\,x_0}(Z,Z^{\prime\prime})
\rho\Big(\frac{4|Z^{\prime\prime}|}{\var}\Big)
T_{g,\,p,\,x_0}(Z^{\prime\prime},Z^{\prime}) 
\kappa_{x_0}(Z^{\prime\prime})
\,dv_{TX}(Z^{\prime\prime})\\+ \cO(p^{-\infty}).
\end{multline} 
\par \textbf{Case $k=0$.} 
By \eqref{toe4.5}, we deduce as in the proof of 
Lemma \ref{toel2.3}, that for $Z,Z^\prime \in T_{x_0}X$, 
$\abs{Z},\abs{Z^{\prime}}<\var/4$, we have
\begin{align} \label{toe4.6}
p^{-n}(T_{f,\,p}\circ T_{g,\,p})_{x_0}(Z,Z^\prime)\cong 
(Q_{0,\,x_0}(f,g)\cP_{x_0})(\sqrt{p}\,Z,\sqrt{p}\,Z^{\prime}) 
+ \bo(1),
\end{align}
where 
\begin{align} \label{toe4.7}
Q_{0,\,x_0}(f,g)=  \cK[Q_{0,\,x_0}(f), Q_{0,\,x_0}(g)]
=f(x_0)g(x_0)\,.
\end{align}
We conclude by Lemma \ref{toet3.4} that 
$T_{f,\,p}\circ T_{g,\,p}-T_{fg,p}=o(1)$, as $p\to\infty$.
\smallskip
\par \textbf{Case $k=1$.} 
By \eqref{toe4.5} and the Taylor formula \eqref{toe2.18} for 
$m=k=1$ we deduce as in the proof of 
Lemma \ref{toel2.3} an estimate analogous to \eqref{toe4.6} 
with $\bo(1)$ replaced by $\bo(p^{-1/2})$, so we obtain 
$T_{f,\,p}\circ T_{g,\,p}-T_{fg,p}=o(p^{-1/2})$, as $p\to\infty$.

\smallskip
\par \textbf{Case $k\geqslant2$.} 
We obtain now that for $Z,Z^\prime \in T_{x_0}X$, 
$\abs{Z},\abs{Z^{\prime}}<\var/4$ and for every 
$m\in\{0,1,\ldots,k\}$ we have
\begin{align} \label{toe4.6a}
p^{-n}(T_{f,\,p}\circ T_{g,\,p})_{x_0}(Z,Z^\prime)\cong 
\sum^m_{r=0}(Q_{r,\,x_0}(f,g)\cP_{x_0})
(\sqrt{p}\,Z,\sqrt{p}\,Z^{\prime})
p^{-\frac{r}{2}} + \mR_{m,p},
\end{align}
where
\[
\mR_{m,p}=
\begin{cases}
\mO(p^{-\frac{m+1}{2}})\,,&\text{if $m\leqslant k-1$}\,,\\
\bo(p^{-\frac{m}{2}})\,,&\text{if $m=k$}\,,
\end{cases}
\]
in the sense of \eqref{toe2.73} and \eqref{toe2.732}. 

Note that for $f,g\in\cC^{\infty}(X,\End(E))$, by 
Lemma \ref{toel2.3} and (\ref{toe4.2}),
we know that 
\begin{align}\label{toe4.7b}
\sum_{r=0}^{\lfloor l/2\rfloor}Q_{l-2r}(C_{r}(f,g))
= Q_{l}(f,g).
\end{align}
As $C_r$ are bidifferential operators with smooth coefficients 
of total degree $2r$ defined in (\ref{toe4.2}), thus for 
 $f,g\in\cC^{k}(X,\End(E))$,
 (\ref{toe4.7b}) still holds for $l\leqslant k$.
 Lemma \ref{toet3.4} and (\ref{toe4.7b}) implies that 
Theorem \ref{toet2.4} holds.
\end{proof}
\begin{corollary}\label{toec2.4}
Let $f,g\in\cC^{k}(X,\End(E))$, $k\in\N$. Then
as $p\to\infty$, in the operator norm sense, we have
\begin{equation}\label{toe2.10d}
T_{f,\,p}\,\circ T_{g,\,p}=T_{fg,p}+R_{0,p}\,,\:\:
R_{0,p}=\begin{cases}
o(1)\,,&\quad\text{for $k=0$}\,,\\
o(p^{-1/2})\,,&\quad\text{for $k=1$}\,.
\end{cases}
\end{equation}
If $k\geqslant2$, then
\begin{equation}\label{toe2.10b}
T_{f,\,p}\,\circ T_{g,\,p}=T_{fg,p}+p^{-1}T_{C_1(f,g),p}+R_p\,,\:\:
R_{p}=\begin{cases}
o(p^{-1})\,,&\quad\text{for $k=2$}\,,\\
o(p^{-3/2})\,,&\quad\text{for $k=3$}\,,\\
O(p^{-2})\,,&\quad\text{for $k=4$}\,.
\end{cases}
\end{equation}
\end{corollary}

By (\ref{toe4.2b}) and (\ref{toe2.10b}), we get
\begin{corollary}
Let $f,g\in\cC^{2}(X)$. Then the commutator of the operators 
$T_{f,\,p}$\,, $T_{g,\,p}$ satisfies
\begin{equation}\label{toe4.4}
\big[T_{f,\,p}\,,T_{g,\,p}\big]=\frac{\sqrt{-1}}{\, p}T_{\{f,g\},\,p}
+R_p\,,\:p\to\infty\,,
\end{equation} 
where $\{ \,\cdot\, , \,\cdot\, \}$ is the Poisson bracket on 
$(X,2\pi \om)$ and $R_p$ satisfies the 
estimates from \eqref{toe2.10b}.
\end{corollary}
The Poisson bracket 
$\{ \,\cdot\, , \,\cdot\, \}$
on $(X,2\pi \om)$ is defined as follows.  For $f, g\in \cC^2 (X)$,
let $\xi_{f}$ be the Hamiltonian vector field generated by $f$, 
which is defined by $2 \pi i_{\xi_{f}}\om=df$. Then 
$\{f, g\}:= \xi_{f}(dg)$.

\section{Asymptotics of the norm of Toeplitz operators}\label{s4}
For $f\in L^{\infty}(X,\End(E))$ we denote the essential 
supremum of $f$ by 
\[{\norm f}_\infty
=\operatorname{ess\,sup}_{x\in X}|f(x)|_{\End(E)}\,.\]
Note that the operator norm $\|T_{f,p}\|$ of $T_{f,p}$ satisfies
\begin{equation}
\norm{T_{f,p}}\leqslant{\norm f}_\infty\,.
\end{equation}
\begin{theorem}\label{toet4.3}
Let $(X,\om)$ be a compact symplectic manifold and let 
$(L,h^L,\nabla^L)\to X$ be a prequantum line bundle 
satisfying \eqref{prequantum}.
Let $(E,h^E,\nabla^E)\to X$ be a twisting Hermitian vector bundle. 
Let $f\in L^{\infty}(X,\End(E))$ and assume that there exists 
$x_0\in X$ such that ${\norm f}_\infty=|f(x_0)|_{\End(E)}$ 
and $f$ is continuous in $x_0$.
Then the norm of $T_{f,\,p}$ satisfies
\begin{equation}\label{toe4.17}
\lim_{p\to\infty}\norm{T_{f,\,p}}={\norm f}_\infty\,.
\end{equation}
\end{theorem}
\begin{proof}
By hypothesis there exist $x_0\in X$ and $u_0\in E_{x_0}$, 
with $|u_0|_{h^{E}}=1$, such that $|f(x_0)(u_0)|={\norm f}_\infty$.
Let us trivialize the bundles $L$, $E$
in normal coordinates over a neighborhood $U$ of $x_0$,
and let $e_L$ be the unit frame of $L$ which trivialize $L$.
In these normal coordinates, we take the parallel transport of 
$u_0$ and obtain a nowhere vanishing section 
$e_E$ of $E$ over $U$.
Denote by $f(x_0)\in\End(E)|_U$ the endomorphism obtained 
by parallel transport of $f(x_0)\in\End(E_{x_0})$.

Let $\delta>0$ be fixed. Since $f$ is continuous, there exists 
$\varepsilon>0$ such that for all $x\in B(x_0,2\varepsilon)$, 
$u\in E_x$, with $|u|_{h^{E}}=1$, we have
\begin{align}\label{eq:4.18}
|f(x)u-f(x_0)u|\leqslant\delta.
\end{align}
Since $x_0$ is fixed let us denote for simplicity $a_i=a_i(x_0)$ 
and set 
\begin{align*}\label{e:peak5}
\begin{split}
&|Z|_a:=\frac12\Big(\sum_{i=1}^na_i|z_i|^2\Big)^{\frac12}\,,
\\
&\rho\in\cC^\infty(X)\,,\:\supp\rho\subset 
B(x_0,\varepsilon)\,,\:\text{$\rho=1$ on $B(x_0,\varepsilon/2)$}.
\end{split}\end{align*}
Define sections
\begin{equation}\label{e:peak}
S^p=S^p_{x_0,\,u_0}=p^{\frac n2}
\sqrt{{\det}_{\C}|\bJ_{x_0}|}
\,P_p\big(\rho e^{-p |Z|_a^2} e_L^{\otimes p}\otimes e_E\big)\in\Ker(D_p)\,.
\end{equation}
Our goal is to prove the following.
\begin{proposition}\label{toe4p}
There exists $C>0$ (independent of $\delta$) such that 
for $p\gg1$,
\begin{equation}
\big\|T_{f,p}S^p-\rho f(x_0)S^p\big\|_{L^2}
\leqslant C\delta\norm{S^p}_{L^2} \,.\label{toe4.18a}
\end{equation}
Moreover, for $p\to\infty$,
\begin{equation}
\big\|\rho f(x_0)S^p\big\|_{L^2}=\|f\|_\infty+O(p^{-\frac12})\,.
\label{toe4.18ab}
\end{equation}
\end{proposition}
We start by showing that $S^p$ are 
\emph{peak sections}\/, i.\,e., satisfies the properties in 
Lemma \ref{lema_peak} below. 
\begin{lemma}\label{lema_peak}
The following expansions hold as $p\to\infty$\,,
\begin{equation}\label{toe4.18g5}
\begin{split}
S^p(Z)=&\,p^{\frac n2}\sqrt{{\det}_{\C}|\bJ_{x_0}|}\,e^{-p|Z|_a^2}
\bigg(1+\sum_{r=1}^kQ_r(\sqrt{p}\,Z)p^{-\frac r2}\bigg)u_0\\
&+\,O\left(p^{\frac{n-k-1}{2}}e^{-C_0\,\sqrt{ p}\,|Z|}
\left(1+\sqrt{p}\,|Z|\right)^{2M}\right)+O(p^{-\infty})\,,
\quad Z\in B(0,\,\varepsilon/2)
\end{split}
\end{equation}
for some constants $C_0,M>0$ and polynomials $Q_r$
with values in $\End(\bE_{x_{0}})$,
\begin{gather}
S^p=O(p^{-\infty})\,,\:\:\text{uniformly on any compact set $K$ 
such that $x_0\not\in K$}\label{toe4.17a}\,,\\
\|S^p\|^2_{L^2}=\int_X|S^p|^2dv_X=1+O(p^{-1})\,.\label{toe4.17b}
\end{gather}
\end{lemma}
\begin{proof} 
By (\ref{e:peak}), we have for $x\in X$,
\begin{equation}\label{e:est_peak}
\begin{split}
S^p(x)
&=p^{\frac n2}\sqrt{{\det}_{\C}|\bJ_{x_0}|}
\int_{B(x_0,\varepsilon)}P_p(x,x')(\rho  e^{-p |Z|_a^2} e_L^{\otimes p}
\otimes e_E)(x')\,dv_X(x')\,.
\end{split}
\end{equation}
We deduce from \eqref{1c3} that for $p\to\infty$,
\begin{equation}\label{e:est_peak1}
S^p(x)=O(p^{-\infty})\,,\:\:\text{uniformly on 
$X\setminus B(x_0,\,2\varepsilon)$}.
\end{equation}
For $Z\in B(0,2\varepsilon)$, by (\ref{atoe2.7}) and 
(\ref{e:est_peak}), we have
\begin{equation}\label{e:est_peak2}
\begin{split}
S^p(Z)&=p^{\frac n2}\sqrt{{\det}_{\C}|\bJ_{x_0}|}
\int_{B(0,\varepsilon)}P_{p,x_0}(Z,Z')\widetilde{\kappa}(Z')  e^{-p |Z'|_a^2} u_0\,dZ'
\,,
\end{split}
\end{equation}
where we have denoted 
$\widetilde{\kappa}=\rho\kappa_{x_{0}}$\,. 
We wish to obtain an expansion of $S^p$ in powers of $p$, 
so we apply Theorem \ref{tue17}.
By \eqref{toe1.9} (see \eqref{toe2.71}) we have
\begin{equation}\label{toe1.9d}
P_{p,x_0}(Z,Z')\kappa_{x_0}^{\frac{1}{2}}(Z)
\kappa_{x_0}^{\frac{1}{2}}(Z')=\sum_{r=0}^k
(J_{r,\,x_0} \cP_{x_0})\,(\sqrt{p}Z,\sqrt{p}Z^{\prime})
p^{n-\frac{r}{2}}
+R_{p,k, x_0}(Z,Z')+O(p^{-\infty})\,.
\end{equation}

By (\ref{atoe2.8}), we write the Taylor expansion of 
$\varphi(Z,Z')=\kappa_{x_0}^{-\frac{1}{2}}(Z)
\kappa_{x_0}^{-\frac{1}{2}}(Z')\widetilde{\kappa}(Z')$ in the form
\begin{equation}\label{toe1.9e}
\begin{split}
&\varphi(Z,Z')=1+\sum_{1<|\alpha|+|\beta|\leqslant k}
\partial^\alpha_Z\partial^\beta_{Z'}\varphi(0,0)
\frac{Z^\alpha}{\alpha!}\frac{Z'^\beta}{\beta!}+O(|(Z,Z')|^{k+1})\\
&=1+\sum_{1<|\alpha|+|\beta|\leqslant k}p^{-(|\alpha|+|\beta|)/2}\,
\partial^\alpha_Z\partial^\beta_{Z'}\varphi(0,0)
\frac{(\sqrt{p}Z)^\alpha}{\alpha!}\frac{(\sqrt{p}Z')^\beta}{\beta!}
+p^{-(k+1)/2}O(|\sqrt{p}(Z,Z')|^{k+1})
\end{split}
\end{equation}
and multiply it with the right-hand side of \eqref{toe1.9d}. 
We obtain in this way an expansion in powers of $p^{1/2}$ 
of $P_{p,x_0}(Z,Z')\widetilde\kappa(Z')$: 
\begin{equation} \label{toe1.9c}
P_{p,\,x_0}(Z,Z^\prime)\widetilde\kappa(Z')=\sum_{r=0}^k
(\widetilde{J}_{r,\,x_0} \cP_{x_0})\,(\sqrt{p}Z,\sqrt{p}Z^{\prime})
p^{n-\frac{r}{2}}
+p^n\widetilde{R}_{p,k}(Z,Z')\,,
\end{equation}
for some polynomials $\widetilde{J}_{r,\,x_0}
\in\End (\bE_{x_0})[Z,Z^{\prime}]$, 
$\widetilde{J}_{0,\,x_0}=I_{\C\otimes E}$, 
and a rest $\widetilde{R}_{p,k}(Z,Z')$ satisfying appropriate 
estimates corresponding to \eqref{toe2.711}.

\noindent
We apply now to $P_{p,x_0}(Z,Z')$ the off-diagonal expansion 
\eqref{toe1.9c} and integrate:
\begin{equation} \label{toe1.9f}
\int_{B(0,\varepsilon)}P_{p,x_0}(Z,Z')\widetilde{\kappa}(Z') e^{-p |Z'|_a^2}u_0\,dZ'
=\sum_{r=0}^kI_rp^{-\frac r2}+ I'_k\,,
\end{equation}
where
\begin{equation} \label{toe1.9g}
\begin{split}
I_r&=\int_{B(0,\varepsilon)}(\widetilde{J}_{r,\,x_0} 
\cP_{x_0})\,(\sqrt{p}Z,\sqrt{p}Z^{\prime})p^{n} e^{-p |Z'|_a^2}
u_0\,dZ'\,,\\
I'_k&=\int_{B(0,\varepsilon)}p^n\widetilde{R}_{p,k}(Z,Z') e^{-p |Z'|_a^2}
u_0\,dZ'\,.
\end{split}
\end{equation}
The norms $Z\mapsto|Z|$ and $Z\mapsto|Z|_a$ are equivalent 
so by the exponential decay of $\cP_{x_0}$ we have,
as $p\to\infty$,
\begin{align}\label{toe1.9h}
I_r=\int_{|Z|_a\leqslant2\varepsilon}(\widetilde{J}_{r,\,x_0} 
\cP_{x_0})\,(\sqrt{p}Z,\sqrt{p}Z^{\prime})p^{n} e^{-p |Z'|_a^2}
u_0\,dZ'+O(p^{-\infty})\,.
\end{align}
We deal first with $I_0$. Let $|Z|_a\leqslant\varepsilon/2$. 
Then by (\ref{toe1.3}), 
\[
\int\limits_{|Z'|_a\leqslant\,2\varepsilon}\cP_{x_0}\,
(\sqrt{p}Z,\sqrt{p}Z^{\prime})p^ne^{-p |Z'|_a^2}\,dZ'=
\int\limits_{\C^{n}}\cP_{x_0}\,(\sqrt{p}Z,\sqrt{p}Z^{\prime})p^n e^{-p |Z'|_a^2}\,
dZ'+O(e^{-Cp}).
\]
By using $\cP_{x_0}$ is a projector operator, we get from 
(\ref{toe1.3}) and  (\ref{toe1.4}),
\begin{align}\label{toe3.0}
\begin{split}
&\int\limits_{\R^{2n}}\cP_{x_0}\,(\sqrt{p}Z,\sqrt{p}Z^{\prime})
p^n e^{-p |Z'|_a^2}\,dZ'\\
&=({\det}_{\C}|\bJ_{x_0}|)^{-1}\int\limits_{\R^{2n}}
\cP_{x_0}\Big(\sqrt{p}\,Z,\sqrt{p}\,Z^{\prime}\Big)
\cP_{x_0}\Big(\sqrt{p}\,Z^{\prime},0\Big)p^n\,dZ'\\
&=({\det}_{\C}|\bJ_{x_0}|)^{-1}\, \cP_{x_0}
\big(\sqrt{p}\,Z,0\big)
=e^{-p|Z|^2_a}\,,
\end{split}
\end{align}
where the first and third equalities follow from \eqref{toe1.3} 
and the second from \eqref{toe1.4}.
We obtain thus
\begin{equation}\label{toe3.1}
\begin{split}
I_0&=\int_{|Z'|_a\leqslant\,2\varepsilon}(\widetilde{J}_{0,\,x_0} 
\cP_{x_0})\,(\sqrt{p}Z,\sqrt{p}Z^{\prime})p^{n}e^{-p |Z'|_a^2}
u_0\,dZ'+O(p^{-\infty})\\
&=e^{-p|Z|^2_a}\,e_{E}+O(p^{-\infty})\,.
\end{split}
\end{equation}
In a similar manner we show that as $p\to\infty$,
\begin{equation}\label{toe3.2}
I_r=e^{-p|Z|^2_a}\,Q_r(\sqrt{p}Z)\,e_{E}+O(p^{-\infty})\,.
\end{equation}
Taking into account the definition of $\widetilde{R}_{p,k}(Z,Z')$, 
\eqref{toe2.711} and \eqref{toe1.9g} we obtain in the same vein,
as $p\to\infty$,
\begin{equation}\label{toe3.3}
I'_k=O\left(p^{\frac{-k-1}{2}}e^{-C_0\,\sqrt{ p}\,|Z|}
\left(1+\sqrt{p}\,|Z|\right)^{2M}\right)+O(p^{-\infty})\,.
\end{equation}
Combining \eqref{toe3.1}-\eqref{toe3.3} we get \eqref{toe4.18g5}.
From \eqref{e:est_peak1} and \eqref{toe4.18g5} we deduce 
immediately \eqref{toe4.17a}.

\noindent
Note that by (\ref{atoe2.7}), (\ref{toe4.17a}), we get
\begin{equation}\label{toe4.18g61}
\begin{split}
\|S^p\|^2_{L^2}&=\int_X\big|S^p(x)\big|^2dv_X(x)
=\int_{B(0,\,2\varepsilon)}\big|S^p(Z)\big|^2\kappa_{x_{0}}(Z)dZ
+O(p^{-\infty})\,.
\end{split}
\end{equation}
By \eqref{atoe2.8} we have
\begin{equation}\label{toe4.18g62}
\begin{split}
&\int_{B(0,\,2\varepsilon)}p^{n}e^{-2p|Z|_a^2}
\kappa_{x_{0}}(Z)\,dZ
=\int_{B(0,\,2\varepsilon\sqrt{p})}e^{-2p|Z|_a^2}
\kappa_{x_{0}}(Z/\sqrt{p})\,dZ\\
&=\int_{\R^{2n}}e^{-\sum_{j=1}^n\frac{1}{2}p\,a_j|Z_j|^2}\,dZ
+O(p^{-1})=\textstyle\prod_{j=1}^n\frac{2\pi}{a_j}+O(p^{-1})\,.
\end{split}
\end{equation}
Further
\begin{equation}\label{toe4.18g32}
\int_{\R^{2n}}\left|p^{\frac n2}
e^{-p|Z|_a^2}Q_r(\sqrt{p}Z)\right|^2\,dZ<\infty\,,
\end{equation}
and
\begin{equation}\label{toe4.18g33}
\int_{B(0,\,2\varepsilon)} \left|p^{\frac{n-k-1}{2}}
e^{-C_0\,\sqrt{ p}|Z|}\left(1+\sqrt{p}\,|Z|\right)^{2M}\right|^2\,dZ
=O(p^{-k-1})\,.
\end{equation}
From \eqref{toe4.18g5}--\eqref{toe4.18g32} we obtain 
\eqref{toe4.17b}. The proof of Lemma \ref{lema_peak}
is completed.
\end{proof}
\begin{lemma}\label{toe4.l3}
We have as $p\to\infty$
\begin{equation}
\text{$T_{f,p}S^p=O(p^{-\infty})$ uniformly on 
$X\setminus B(x_0,2\varepsilon)$}\,.
\end{equation}
\end{lemma}
\begin{proof}
Due to Lemma  \ref{toet2.1} and \eqref{toe4.17a},
as $p\to\infty$, we have
\begin{equation}
\begin{split}
T_{f,p}S^p(x)
&=\int_{B(x_0,\varepsilon)} T_{f,\,p}(x,x')S^p(x')\,dv_X(x')
+O(p^{-\infty})\\
&=O(p^{-\infty})\,,
\end{split}
\end{equation}
uniformly for $x\in X\setminus B(x_0,2\varepsilon)$.
\end{proof}
\begin{lemma}\label{toe4.l4}
As  $p\to\infty$, we have \begin{equation}\label{toe4.18b}
\int_{X\setminus B(x_0,2\varepsilon)}\big|T_{f,p}S^p
-\rho f(x_0)S^p\big|^2\,dv_X
=O(p^{-\infty})\,.
\end{equation}
\end{lemma}
\begin{proof}
This follows immediately from \eqref{toe4.17a} and 
Lemma \ref{toe4.l3}.
\end{proof}
\begin{lemma}\label{toe4.l5} For $p\gg1$,
we have 
\begin{equation}\label{toe4.18c}
\int_{B(x_0,2\varepsilon)}\big|T_{f,p}S^p
-\rho f(x_0)S^p\big|^2\,dv_X
\leqslant C^2\delta^2\|S^p\|^2_{L^2}\,.
\end{equation}
\end{lemma}
\begin{proof} We have $P_pS^p=S^p$, since $S^p\in\Ker(D_p)$. 
    Thus $T_{f,p}S^p=P_p(fS^p)$.
Hence
\[
(T_{f,p}S^p)(x)=\int_X P_p(x,x')f(x')S^p(x')\,dv_X(x')\,.
\]
Let us split
\begin{equation}\label{toe4.18d}
(T_{f,p}S^p)(x)-(\rho f(x_0)S^p)(x)=g_p(x)+h_p(x)\,,\\
\end{equation}
where 
\begin{equation*}
\begin{split}
&g_p(x):=\int_XP_p(x,x')\big[f(x')-\rho f(x_0)\big]S^p(x')\,
dv_X(x')\,,\\
&h_p(x):=\int_XP_p(x,x')\rho f(x_0)S^p(x')\,dv_X(x')
-(\rho f(x_0)S^p)(x)\,.
\end{split}
\end{equation*}
Set 
\[
R_{p,1,x_0}(Z,Z'):=P_{p,x_0}(Z,Z')\kappa_{x_0}(Z')
-p^n(J_0\cP_{p,x_0})(\sqrt{p}Z,\sqrt{p}Z')
\,.
\]
By \eqref{toe2.711} we have
\begin{equation}\label{toe4.18e}
\left|R_{p,1,x_0}(Z,Z^\prime)\right|\leqslant  \,C\,p^{n-\frac12}
(1+\sqrt{p}\,|Z|+\sqrt{p}\,|Z^{\prime}|)^M \,
e^{-C_0\,\sqrt{ p}\,|Z-Z^{\prime}|}+O(p^{-\infty})\,.
\end{equation}
Set 
\begin{align}\label{toe4.20}  \begin{split}
I_{1,p}(Z)&=\int_{B(0,2\varepsilon)}R_{p,1,x_0}(Z,Z^\prime)
\rho f(x_0)S^p(Z')\kappa(Z')\,dZ'\,,\\
S^p_0&=\sqrt{\det|\bJ_{x_0}|}p^{\frac n2}e^{-p|Z|_a^2}\,u_0,\\
I_{2,p}(Z)&=\int_{B(0,2\varepsilon)}p^n(J_{0,x_0}\cP)
(\sqrt{p}Z,\sqrt{p}Z')\rho f(x_0)(S^p-S^p_0)(Z')\,dZ'.
\end{split}\end{align}
We have
\begin{align}\label{toe4.21a}
h_p(Z)=I_{1,p}(Z)+I_{2,p}(Z)+\rho f(x_0)\big(S^p-S^p_0\big)(Z')
+O(p^{-\infty})\,.
\end{align}

Estimates \eqref{toe4.18g5} and \eqref{toe4.18e} entail
\begin{equation}\label{toe4.21}
\big|I_{1,p}(Z)\big|
\leqslant \,C\,p^{\frac{n-1}{2}}
(1+\sqrt{p}\,|Z|)^{2M} \,
e^{-C_0\,\sqrt{ p}\,|Z|}+O(p^{-\infty})\,.
\end{equation}

By \eqref{toe4.18g5},
\begin{equation}\label{toe4.22}
\big|I_{2,p}(Z)\big|\leqslant p^{\frac{n-1}{2}}\,e^{- C\sqrt{p}\,|Z|}\,.
\end{equation}
By  \eqref{toe4.18g5},  \eqref{toe4.18g32}, \eqref{toe4.21a},
\eqref{toe4.21} and \eqref{toe4.22},  we obtain as $p\to \infty$,
\begin{equation}\label{toe4.23}
\begin{split}
&\bigg(\int_{B(x_0,2\varepsilon)}|h_p(x)|^2\,
dv_X(x)\bigg)^{\frac12}\leqslant
\bigg(\int_{B(0,2\varepsilon)}|I_{1,p}(Z)|^2\,
dv_X(Z)\bigg)^{\frac12}+O(p^{-\infty})\\
&+\bigg(\int_{B(0,2\varepsilon)}|I_{2,p}(Z)|^2\,
dv_X(Z)\bigg)^{\frac12}+
\bigg(\int_{B(0,2\varepsilon)}\big|\rho f(x_0)
\big(S^p-S^p_0\big)(Z)\big|^2\,dv_X(Z)\bigg)^{\frac12}\\
&=O(p^{-\frac12})\,.
\end{split}
\end{equation}
Moreover for $g_{p}$ from \eqref{toe4.18d}, we get by 
(\ref{eq:4.18}) that for 
$Z\in T_{x_{0}}X$, $|Z|\leqslant 2\varepsilon$ we have
\begin{equation*}
\begin{split}
\big|g_p(Z)\big|\leqslant
\delta\int_{B(0, 4\varepsilon)}\widetilde{g}(Z,Z')\,dv_X(Z')
+O(p^{-\infty})\,,
\end{split}
\end{equation*}
where
\[
\widetilde{g}(Z,Z')=\big(p^ne^{-C\sqrt{p}|Z-Z'|}
(1+\sqrt{p}Z+\sqrt{p}Z')^M\big)
p^{\frac n2}e^{-p|Z|_a^2}(1+\sqrt{p}|Z'|)^{M'}\,,
\]
hence
\begin{equation}\label{toe4.24}
\begin{split}
\big|g_p(Z)\big|\leqslant C\delta\big(p^{\frac n2}\,
e^{-C\sqrt{p}\,|Z|}+O(p^{-\infty})\big)\,.
\end{split}
\end{equation}
From \eqref{toe4.17b} and \eqref{toe4.24} we infer
\begin{equation}\label{toe4.20ab}
\int_{B(x_0,2\varepsilon)}|g_p(x)|^2\,dv_X(x)
\leqslant C_1^2\delta^2\|S^p\|_{L^2}^2\,. 
\end{equation}
Now \eqref{toe4.18d}, \eqref{toe4.23} and \eqref{toe4.20ab} 
yield the desired estimate \eqref{toe4.18c}.
\end{proof}

\noindent
Lemmas \ref{toe4.l4} and \ref{toe4.l5} yield \eqref{toe4.18a}. 
\comment{Let us now prove \eqref{toe4.18ab}. We have
\[
\begin{split}
\big\|\rho f(x_0)S^p\big\|^2_{L^2}&=\int_{B(0,\varepsilon)}
\big|\rho f(x_0)S^p(Z)\big|^2\kappa(Z)\,dZ\\
&=|f(x_0)|^2\int_{B(0,\varepsilon)}\big|\rho S^p(Z)\big|^2
\kappa(Z)\,dZ\,.
\end{split}
\]
Similar to \eqref{toe4.18g62} we have }
By (\ref{atoe2.8}), (\ref{toe4.18g5}) and (\ref{toe4.20}),
similar to \eqref{toe4.18g62}, we have for $p\to\infty$,
\begin{equation}\label{eq:4.24}
\begin{split}
\big\|\rho f(x_0)S^p\big\|^2_{L^2}&
=\int_{B(0,\varepsilon)}\big|f(x_0)\rho S^p(Z)\big|^2
\kappa(Z)\,dZ\\
&=\int_{B(0,\varepsilon/2)}\big|f(x_0)S^p_{0}(Z)\big|^2\,dZ
+O(p^{-1})\\
&= |f(x_{0}) u_{0}| ^2 
\int_{B(0,\varepsilon/2)} p ^n {\det}_{\C}|\bJ_{x_{0}}| e ^{-2p 
|Z|_{a}^2}\,dZ +O(p^{-1})\\
&=\|f\|_{\infty}^2+O(p^{-1})\,.
\end{split}
\end{equation}
This completes the proof of Proposition \ref{toe4p}.
\end{proof}
\begin{remark}
If we improve the regularity of the section $f$ in 
Theorem \ref{toet4.3}, the convergence speed in \eqref{toe4.17}
improves accordingly (by improving Lemma  \ref{toe4.l5}):
\\[2pt]
(a) If $f\in\cC^1(X,\End(E))$ then there exists $C>0$ such that
\[
\norm{f}_{\infty}-\frac{C}{\sqrt{p}}\leqslant\norm{T_{f,p}}
\leqslant\norm{f}_{\infty}\,.
\] 
The estimate does not improve even if $f$ is a function.
\\[2pt]
(b) Assume that in Theorem \ref{toet4.3}, $(X,J,\omega)$ is 
K\"ahler, $(L,h^L,\nabla^L)\to X$ is a prequantum holomorphic 
line bundle (where $\nabla^L$ is the Chern connection)
satisfying \eqref{prequantum}  and $(E,h^E,\nabla^E)\to X$ is 
a holomorphic Hermitian vector bundle
with the Chern connection $\nabla ^E$. The K\"ahler assumption 
implies that $J_1(Z,Z')=0$ (cf. \cite[(4.1.102)]{MM07}). 
Using (\ref{atoe2.8}) and 
\[\int_{\C^n}e^{-\frac{\pi}{2} p |Z|^2} Z_j\,dZ=0,\]
we deduce that for $f\in\cC^2(X,\End(E))$
there exists $C>0$ such that
\[
\norm{f}_{\infty}-\frac{C}{p}\leqslant\norm{T_{f,p}}
\leqslant\norm{f}_{\infty}\,.
\] 
For $f\in\cC^1(X,\End(E))$ we have the same estimate as in (a), 
which cannot be improved.
\end{remark}
\begin{remark}
Theorem \ref{toet4.3} holds also for large classes of non-compact 
manifolds, see \cite[\S7.5]{MM07}, \cite[\S5]{MM08b}, 
\cite[\S2.8]{MM11}.
\end{remark}


\section{How far is $T_{f,p}$ from being self-adjoint or 
multiplication operator}\label{s5}

In this section we continue to work in the setting of Section 
\ref{sec:2.1}.
Let $(X,\omega)$ be a $2n$-dimensional connected compact
symplectic manifold, 
and $(L,h^L,\nabla^L)\to X$ be a prequantum line bundle satisfying 
\eqref{prequantum}.
We assume in the following that the vector bundle $E$ is trivial 
of rank one ($E=\C$).
To avoid lengthy formulas let us denote
\begin{align}\label{eq:4.1}
\hp=\ker(D_p).
\end{align}

Let $\cC^{0}(X)$ denote the space of continuous complex-valued 
functions on $X$. We shall denote by $\cC^{0}(X,\R)$ 
the space of continuous real-valued functions on $X$. 
For $f,g\in\cC^{0}(X)$ set 
\begin{align}\label{eq:4.2}
\big\langle f,g\big\rangle=\int_X f(x)\overline{g(x)}\:
\frac{\omega^n}{n!}\, \cdot
\end{align}

Let $L^2(X)$ be the completion of $\cC^{0}(X)$ with respect 
to the norm $\|f\|=\sqrt{\langle f,f\rangle }$ and let $L^2(X,\R)$ 
be the subspace of $L^2(X)$ that consists of 
(equivalence classes of) 
real-valued functions. 
By a slight abuse of notation, we denote by 
$\C\subset\cC^{0}(X)$ the $1$--\,dimensional subspace of 
$\cC^{0}(X)$ that consists of constant functions.   
Note: $L^2(X,\R)$ and $\C$ are closed subspaces of $L^2(X)$. 
For $f\in L^2(X)$ the orthogonal projection 
of $f$ onto $\C$ (with respect to the inner product 
$\langle \cdot,\cdot\rangle $) 
is the constant function 
\begin{align}\label{eq:4.3}
\fint_X f\:\frac{\omega^n}{n\,!}
:=\frac{\big\langle f,1\big\rangle }{\big\langle 1,1\big\rangle }=
\frac{1}{\vol (X)}\int_X f\:\frac{\omega^n}{n!}\, 
\quad\text{    with }  \vol (X)= \int_X \frac{\omega^n}{n!} ,
\end{align}
and the orthogonal projection of $f$ onto $L^2(X,\R)$ 
(with respect to the inner product $\Re\langle \cdot,\cdot\rangle $) 
is $\Re(f)$. 

Denote by $(\,\cdot\,,\cdot)_{HS}$ the Hilbert-Schmidt inner 
product 
on $\End (\hp )$: for $A,B\in \End (\hp )$ 
$$
(A,B)_{HS} =\tr (AB^*),
$$ 
where $B^*$ is the adjoint of $B$.
Note that the operator norm 
does not exceed the Hilbert-Schmidt norm.
The inner product on the underlying real vector space 
$\End _{\R}(\hp )$ is given by 
$$
(A,B)_{\R} =\Re \ \tr (AB^*).
$$ 
Denote by $\hph$ the subspace of $\End _{\R}(\hp )$ that 
consists of self-adjoint (Hermitian) operators. Denote by $\hpc$ 
the subspace of $\End (\hp )$ that consists of constant 
multiples of the identity operator. 

We shall use $\dist (v,V)$ to denote the distance between 
an element $v$ of a normed vector space and a closed subspace 
$V$. For example, for $f\in L^2(X)$ 
\begin{equation}\label{toe4.9}
\dist (f,\C)^2=\int_X |f|^2\:\frac{\omega^n}{n!}
- \frac{1}{\vol (X)}\bigg|\int_X f\:\frac{\omega^n}{n!}\bigg|^2\,
\cdot
\end{equation}
It is clear that for any $f\in L^{\infty}(X)$ we have
\begin{equation}\label{toe4.10}
T_{f,p}^*=T_{\overline{f},p}\,,
\end{equation}
hence
for $f\in L^{\infty}(X,\R)$ the operator 
$T_{f,p}$ is self-adjoint. We denote $M_{f}$
the pointwise multiplication by $f$.
On the other hand, if $f\in L^{\infty}(X)$ is constant then 
$T_{f,p}=M_f$. 

Theorem \ref{bptheorem1}, stated below, addresses, 
informally speaking, 
the following issues: given $f\in\cC^{0}(X)$, 

(1) how far $f$ is from being real-valued should be related to 
how far 
$T_{f,p}$ is from being self-adjoint (in $\End _{\R}(\hp )$), 

(2) how far $f$ is from being constant 
should be related to how far $T_{f,p}$ is from being 
a constant multiple of the identity operator (in $\End (\hp )$).  
 
\begin{theorem}
\label{bptheorem1}
Let $f\in L^{\infty}(X)$. Write
\begin{align}\label{toe4.11a}\begin{split}
p^{-n}\big[\dist(T_{f,p}\,,\hph)\big]^2
&=\big[\dist (f,L^2(X,\R))\big]^2+R_{1,p}\,, \\
p^{-n}\big[\dist(T_{f,p}\,,\hpc )\big]^2&
=\big[\dist(f,\C)\big]^2 +R_{2,p}\,.
\end{split}
\end{align}
Then $R_{i,p}$\,, $i=1,2$\,, satisfy as $p\to\infty$
\begin{equation}\label{toe4.11}
R_{i,p}=\begin{cases}
o(1)\,,&\quad \text{uniformly on } f\in \cA_{\infty} ^0\, , \\
O(p^{-1/2})\,,&\quad \text{uniformly on } f\in \cA_{\infty} ^1\, , \\
O(p^{-1})\,,&\quad\text{uniformly on } f\in \cA_{\infty} ^2\, .
\end{cases}
\end{equation}
\end{theorem}
\begin{proof}
We consider first the case $i=1$. Let $A\in \End (\hp )$. 
The orthogonal projection of $A$ 
onto $\hph$ is $\frac{1}{2}(A+A^*)$. We have: 
\begin{align}\label{toe4.12a}
\begin{split}
\big[\dist (A,\hph )\big]^2 &= \big(\tfrac{1}{2}(A-A^*),
\tfrac{1}{2}(A-A^*)\big)_{\R }=
\tfrac{1}{4} \tr \big[(A-A^*)(A^*-A)\big]
\\
&=\tfrac{1}{4}\big[-\tr (A^2)-\tr ((A^*)^2)+\tr (AA^*)
+\tr (A^* A)\big] .
\end{split}
\end{align}
We apply the previous formula for $A=T_{f,p}$ by
using Theorem \ref{toet2.5} and \eqref{toe4.10},
we get as $p\to\infty$  
\begin{align}\label{toe4.13a}
\begin{split}
p^{-n}\big[\dist (A,\hph )\big]^2&=
\tfrac14\int_X \big(-f^2-\overline{f}^{\,2}+
2f\overline{f}\,\big)\,\frac{\omega^n}{n\,!}+R_{1,p}\\ 
&= \int_X \big|\Im(f)\big|^2\,\frac{\omega^n}{n\,!}+R_{1,p}\,,
\end{split}
\end{align}
where $R_{1,p}$ satisfies \eqref{toe4.11}.
This proves the assertion of the Theorem for $i=1$, since 
\begin{align}\label{toe4.14a}
\big[\dist(f,L^2(X,{\R}))\big]^2=\int_X \big|f-\Re(f)\big|^2\:
\frac{\omega^n}{n\,!}\,\cdot
\end{align}
We consider now the case $i=2$.
For $A\in\End(\hp)$ the orthogonal projection 
of $A$ onto $\hpc$ is $\alpha  {\rm Id}_{\mH_{p}}$, 
and $\alpha =\frac{1}{\dim\hp}\tr(A)$. Therefore  
\begin{equation}\label{toe4.12}
\begin{split}
\big[\dist(A,\hpc )\big]^2&
=(A-\alpha {\rm Id}_{\mH_{p}},A-\alpha {\rm Id}_{\mH_{p}})
=\tr\big[(A-\alpha {\rm Id}_{\mH_{p}})
(A^*-\bar{\alpha }{\rm Id}_{\mH_{p}})\big]\\
&=\tr (AA^*-\alpha A^*-\bar{\alpha }A
+\alpha \bar{\alpha }{\rm Id}_{\mH_{p}})\\
&=\tr (AA^*)-\frac{1}{\dim\hp}\tr (A)\tr (A^*)\,,
\end{split}
\end{equation}
since $\tr({\rm Id}_{\mH_{p}})=\dim\hp$.
Note that by the Atiyah-Singer index formula \eqref{as}, we have 
\begin{equation}\label{toe4.14}
\dim\hp=p^n\vol(X)+O(p^{n-1})\,.
\end{equation}
We apply formula \eqref{toe4.12} for $A=T_{f,p}$ by
using Theorem \ref{toet2.5}, \eqref{toe4.9}, \eqref{toe4.10} 
and \eqref{toe4.14} to get 
\begin{align}\label{toe4.15a}
\begin{split}
p^{-n}\big[\dist(T_{f,p}\,,\hpc)\big]^2&=\int_X |f|^2\:
\frac{\omega^n}{n\,!}- \frac{1}{\vol (X)}\bigg|\int_X f\:
\frac{\omega^n}{n\,!}\bigg|^2+R_{2,p}\\
&=\big[\dist (f,\C)\big]^2+R_{2,p}\,,
\end{split}
\end{align}
where $R_{2,p}$ satisfies \eqref{toe4.11} as $p\to\infty$.
\end{proof}

It is also intuitively clear that how far $f$ is from being constant
(i.\,e.\ how far $df$ is from zero) should be related 
to how far $T_{f,p}$ is from $M_fP_p$.
This is addressed in the next Proposition. The main point here is 
to estimate the Hilbert-Schmidt norm of the difference 
$T_{f,p}-M_fP_p$ \emph{uniformly} for $f\in\cC^1(X)$.
\begin{proposition}\label{bpprop} 
We suppose that 
$g ^{TX}(\cdot,\cdot)= \omega(\cdot,J\cdot)$.
Let $\lambda>0$ be the lowest positive eigenvalue of 
the Laplace operator $\Delta_{g^{TX}}$ acting on functions. 
Then  for any $\var>0$, there exists $p_{0}>0$
such that for any $p\geqslant p_{0}$, $f\in\cC^1(X)$, 
we have
\begin{equation}\label{eq:4.10a}
p^{-n}\|T_{f,p}-M_fP_p\|_{HS}^2
\leqslant\lambda^{-1} (1+\varepsilon)\|df\|^2_{L^2}. 
\end{equation}
\end{proposition}
\begin{proof}
Denote for simplicity $\Delta=\Delta_{g^{TX}}$. 
The Hodge decomposition of $f$ has the form $f=f_1+f_2$, 
where $f_1\in\Ker(\Delta)$ is the harmonic component of $f$ 
and $f_2\in(\Ker\Delta)^\perp$. We have $\Ker(\Delta)=\C$ 
and $f_1=\fint_Xf\,dv_X$. 
Moreover, 
\begin{align}\label{eq:4.11}
\|f_2\|_{L^2}^2\leqslant\lambda^{-1}\|df_2\|_{L^2}^2
=\lambda^{-1}\|df\|_{L^2}^2.
\end{align}
Now
\begin{equation}\label{eq:4.12}
T_{f,p}-M_fP_p = T_{f_{2},p}-M_{f_{2}}P_p\,.
\end{equation}
As $\tr[T_{f_{2},p}T_{f_{2},p} ^*]>0$, we get by using 
\eqref{bk2.4a} and \eqref{eq:4.12} that
\begin{equation}\label{eq:4.13}
\begin{split}
\|T_{f,p}-M_fP_p\|_{HS}^2 
&= \tr \Big[ T_{f_{2},p}T_{f_{2},p} ^*
+f_{2} P_{p} \ov{f}_{2} -  T_{f_{2},p} \ov{f}_{2}
- f_{2} P_{p} \ov{f}_{2} P_{p} \Big]\\
&= \tr \Big[ P_{p} f_{2}\ov{f}_{2} P_{p} 
-T_{f_{2},p}T_{f_{2},p}  ^*\Big]
\leqslant  \tr \Big[ P_{p} |f_{2}| ^2 P_{p}\Big]\\
&=\sum_{i=1} ^{d_{p}} \| f_{2} S^P_{i}\|_{L^2} ^2 
= \int_{X}  |f_{2}(x)| ^2 \tr [P_{p}(x,x)] dv_{X}(x).
\end{split}
\end{equation}
By  the argument after (\ref{toe1.9b}), 
for any $\varepsilon>0$, there exists $p_{0}>0$
such that for $p\geqslant p_{0}$, we have 
\begin{align}\label{eq:4.14}
\int_{X}  |f_{2}(x)| ^2 \tr [P_{p}(x,x)] dv_{X}(x)
\leqslant (1+\varepsilon)  p  ^n \|f_2\|_{L^2}^2.
\end{align}
By \eqref{eq:4.11},  \eqref{eq:4.13} and \eqref{eq:4.14}, 
we get \eqref{eq:4.10a}.
\end{proof}
\begin{remark}
The results in this paper hold in particular in the case of the K\"ahler 
quantization. Let us assume that $(X,J,g^{TX})$ is a compact K\"ahler manifold (i.\,e., $J$
is integrable and $g^{TX}(u,v)=\om(u,Jv)$ for 
$u,v\in TX$). Assume moreover that the bundles $L$ and $E$ 
are holomorphic and $\nabla^L$, $\nabla^E$ are the 
Chern connections. Then by Remark \ref{R:kaehler},
the quantum space are the spaces of global holomorphic sections 
of $L^p\otimes E$. We can even dispense of the K\"ahler condition 
$g^{TX}(u,v)=\om(u,Jv)$ for 
$u,v\in TX$, see Remark \ref{R:kaehler}.

To illustrate the kind of results we obtain in the K\"ahler case let us 
formulate the following special case of of Proposition \ref{bpprop}.
\end{remark}
\begin{proposition}\label{bpprop1} 
Assume that $(X,J,g^{TX})$ is a compact K\"ahler manifold and the 
bundles $L$ and $E$ are holomorphic. 
Let $\lambda>0$ be the lowest positive eigenvalue of 
the Kodaira Laplace operator 
$\overline{\partial}^{\,*}\overline{\partial}$ acting on functions. 
Then  for any $\var>0$, there exists $p_{0}>0$
such that for any $p\geqslant p_{0}$, $f\in\cC^1(X)$, 
we have
\begin{equation}\label{toe4.151}
p^{-n}\|T_{f,p}-M_fP_p\|_{HS}^2
\leqslant\lambda^{-1} (1+\varepsilon)\|\overline{\partial}f\|^2_{L^2}. 
\end{equation}
\end{proposition}
\smallskip
\begin{ack} We are grateful to Brian Hall and Armen Sergeev 
    for sending us interesting references treating 
    Berezin-Toeplitz operators with non-smooth symbols.
\end{ack}
\medskip


\def\cprime{$'$} \def\cprime{$'$}
\providecommand{\bysame}{\leavevmode\hbox to3em{\hrulefill}\thinspace}
\providecommand{\MR}{\relax\ifhmode\unskip\space\fi MR }
\providecommand{\MRhref}[2]{%
  \href{http://www.ams.org/mathscinet-getitem?mr=#1}{#2}
}
\providecommand{\href}[2]{#2}

\end{document}